\documentclass[12pt]{article}
\usepackage{FrankTall}
\usepackage[usenames,dvipsnames]{xcolor}
\usepackage{xspace}
\usepackage{perpage}
\MakePerPage{footnote}

\usepackage[norule,hang]{footmisc}
\newcommand{\arhan}{Arhangel'ski\u{\i}\xspace}

\newcommand{\frech}{Fr\'echet-Urysohn\xspace}
\newcommand{\groth}{Gro\-then\-dieck\xspace}
\newcommand{\lind}{Lindel\"of\xspace}

\newcommand\blfootnote[1]{%
  \begingroup
  \renewcommand\thefootnote{}\footnote{#1}%
  \addtocounter{footnote}{-1}%
  \endgroup
}

\title{$C_p$-Theory for Model Theorists}
\author{Clovis Hamel${}^{1,2}$ \ and Franklin D. Tall${}^{2}$}
\date{\today}

\begin{document}

\maketitle
\blfootnote{${}^1$ The first author is an NSERC Vanier Scholar.} \blfootnote{${}^2$ Research supported by NSERC Grant A-7354.}

\blfootnote{\textit{2010 Mathematics Subject Classification}. 03C45, 03C75, 03C95, 03C98, 54C35, 46B99}

\blfootnote{\textit{Key words and phrases}. Tsirelson's space, Gowers' problem, explicitly definable Banach spaces, $C_p$-theory, model-theoretic stability, definability, double limit conditions, Grothendieck spaces.}

\abstract{We present applications of $C_p$-theory, the branch of general topology concerned with spaces of real-valued continuous functions, to model theory, mostly in the context of continuous logics. We include $C_p$-theoretic results and proofs in a self-contained way for model theorists who are not familiar with the techniques of this field. We further generalize some results of Casazza and Iovino, and of the authors, involving the definability of Banach spaces including isomorphic copies of $c_0$ or $\ell^p$, after a problem posed by Odell and Gowers.}

\section{Introduction}
$C_p$-theory is a subfield of general topology. Given a topological space $X$, $C_p(X)$ is the collection of continuous real-valued functions on $X$, considered as a subspace of $\reals^X$. $C_p$-theory is concerned with the interrelations between the topology on $X$ and the topology on $C_p(X)$. First identified as a distinct field by A.~V.~\arhan, it has grown to the point of having a monograph \cite{ArhangelskiiFunction}, a four-volume problem book \cite{Tkachuk2011}, and many research papers devoted to it. Perhaps the most famous theorem is that of A.~Grothendieck \cite{Grothendieck1952}, proved long before \arhan's work:

\begin{thm}
	If $X$ is countably compact and $A \subseteq C_p(X)$, then the closure of $A$ in $C_p(X)$ is compact if and only if $A$ is countably compact in $C_p(X)$.
\end{thm}

The usefulness of Grothendieck's work in model theory, especially with regard to stability was recognized by J.~Iovino \cite{Iovino1999}, I.~Ben Yaacov \cite{BenYaacov2014} and A.~Pillay \cite{Pillay2017}, all in the context of logics satisfying the compactness theorem.

When Iovino described his use of Grothendieck's theorem in his solution with Casazza \cite{Casazza} of a problem of Gowers \cite{Gowers1995} to the second author, the latter realized it could be considerably generalized by using the methods and results of $C_p$-theory. This was done in the recent \cite{Hamel}, with the goal of introducing model theory to $C_p$-theorists unfamiliar with mathematical logic in general, and with model theory in particular. Here we instead target a model-theoretic audience unfamiliar with $C_p$-theory and not very knowledgeable of general topology. Our goal is to present enough $C_p$-theory so that model theorists can use it, as well as to again give an example of its usefulness by further generalizing the results of \cite{Casazza} and \cite{Hamel}. 

It is worth mentioning that \cite{Casazza} contained errors in its model-theoretic component that have been reworked in \cite{CDI}. We refer to definitions in \cite{Casazza} and the results from analysis that appear in it, which are correct. At the end of Section 3, we discuss the different approaches from \cite{Casazza} and \cite{CDI}, and how they have influenced our work. 

We have omitted some $C_p$ proofs which are very easy or not particularly germane to the main line of our exposition. On the other hand, we have included some $C_p$ proofs which are long and not easy, because the techniques are central or instructive or not easy to find in the literature. We have put these in an Appendix. We will obtain our undefinability results using special cases which are easy to prove, but thought we should include the general cases for those interested.

The main application of the results occurs with regard to Gowers' problem \cite{Gowers1995} which asks whether every infinite dimensional explicitly definable Banach space includes an isomorphic copy of either $c_{0}$ or $\ell^p$. All the definitions are given below but a model theorist would probably recognize at once that this question intersects her field as it is mainly about definability. There are two fundamental facts that motivate Gowers' problem:

\begin{enumerate}[label=(\arabic*)]
	\item All classical Banach spaces include a copy of $c_{0}$ or $\ell^p$.
	\item B.~Tsirelson \cite{Tsirelson1974} constructed a Banach space not including copies of $c_{0}$ or $\ell^p$. Such a space is regarded as the ``first truly nonclassical Banach space'' \cite{Gowers1995}.
\end{enumerate}

An interesting characteristic is that the norm of Tsirelson's space appears on both sides of the equality in its definition; this has become what analysts call an \emph{implicitly definable} norm, as opposed to \emph{explicitly definable} norms which are simply defined by a formula. (Of course one has to prove that such a norm is well-defined; this can be found in the book \cite{Casazza1989} devoted to the Tsirelson space.) After Tsirelson, many other pathological Banach spaces were constructed using similar techniques. Of course, the first problem to face is to formalize the question, and E.~Odell says in his book \cite{Odell2002} that this is a logician's task. The natural logics to approach this problem are \emph{continuous logics} as they provide ample room for approximation and the analyst's $\epsi$-play. \cite{Casazza} and \cite{CDI} presented both a formulation and a solution of this problem in the context of compact continuous logics, i.e.,~continuous logics which satisfy the Compactness Theorem. We shall present here generalizations of some results obtained in \cite{Casazza} and \cite{CDI} to a wide class of logics with the help of $C_p$-theory, in particular to non-compact, countable fragments of continuous $\mcal{L}_{\omega_1, \omega}$.

We shall allow ourselves to repeat well-known $C_p$-theoretic results and proofs from the literature for a reasonably self-contained exposition so that this work can serve as an introduction for model theorists wishing to learn $C_p$-theory.

Perhaps our title - a reflection of  \cite{Hamel} - is too ambitious, given the four-volume \cite{Tkachuk2011}. A more modest one would be ``$C_p$-theory and Definability.''

\renewcommand*{\thefootnote}{\fnsymbol{footnote}}
\section{Preliminaries in $C_p$-theory}
\subsection{Some Basic Results}
We shall assume basic knowledge of general topology such as can be found in any of the standard texts for a first graduate course, such as \cite{Willard2012}. Nonetheless, in footnotes we shall remind the reader of definitions she may have forgotten. Unless otherwise stated, we assume that all the topological spaces we consider are infinite, completely regular, and Hausdorff.

\begin{defi}
	Let $X$ be a completely regular\footnote{A space $X$ is \emph{completely regular} if for each closed set $F$ and each point $x$ not in $F$ there is a continuous real-valued $f: X \to \reals$ such that $f(x) = 0$ and $f(F) = \{1\}$.} Hausdorff topological space. $C_p(X)$ is the space of real-valued continuous functions on $X$, endowed with the subspace topology inherited from the product topology on $\reals^X$. Thus, the basic open sets of $C_p(X)$ are just the basic open sets of $\reals^X$ intersected with $C_p(X)$, i.e.
	\begin{linenomath*}
	\[W(g;x_0, \ldots, x_n; \epsi) = \st{f \in C_p(X)}{(\forall i \leq n)(\abs{f(x_i) - g(x_i)} < \epsi)}.\]
	\end{linenomath*}
\end{defi}

\begin{rmk*}
	It follows from the form of the basic open sets of $\reals^X$ and complete regularity that $C_p(X)$ is dense in $\reals^X$ (given any function $g\in\reals^X$ and a basic open set in $\reals^X$ about $g$, say $U=U(g;x_0, \ldots, x_n; \epsi)=\st{f \in \reals^X}{(\forall i \leq n)(\abs{f(x_i) - g(x_i)} < \epsi)}$, we can find a continuous function $h\in C_p(X)$ such that $(\forall i \leq n)h(x_i)=f(x_i)$ and so $h\in U$). We shall occasionally deal with $C_p(X, [0, 1])$ and $C_p(X, 2)$, with the obvious meanings. Notice that each of these is a closed subspace of $C_p(X)$. Since the main properties we deal with are closed-hereditary, if we prove results for $C_p(X)$, they apply also to the other two cases. In particular, this means that our results can be applied to first order logic, in which $C_p(X, 2)$ is relevant.
\end{rmk*}

\begin{lem}
\label{lem:rstrct}
	Let $X$ be any topological space, $Y \subseteq X$, and $\pi_Y: C_p(X) \to C_p(Y)$ the restriction map given by $\pi_Y(f) = \rstrct{f}{Y}$, $f$ restricted to $Y$. Then:
	\begin{enumerate}[label=(\roman*)]
		\item $\pi_Y$ is continuous.
		\item $\clsr{\pi_Y\left(C_p(X)\right)} = C_p(Y)$.
		\item If $Y \subseteq X$ is dense, then $\pi_Y: C_p(X) \to \pi_Y\left(C_p(X)\right)$ is a continuous bijection.
		\item If $Y \subseteq X$ is closed, then $\pi_Y: C_p(X) \to \pi_Y\left(C_p(X)\right)$ is an open map.
	\end{enumerate}
\end{lem}

\begin{proof}\mbox{}
\begin{enumerate}[label=(\roman*)]
	\item Note that $\pi_Y$ is continuous as the pre-image of a basic open set of $C_p(Y)$ is a basic open set of $C_p(X)$.

	\item Take $f \in C_p(Y)$ and an open neighbourhood 
	\begin{linenomath*}
	\[U = W(f; x_0, \ldots, x_n; \epsi) \subseteq C_p(Y).\]
	\end{linenomath*}
	Since $X$ is completely regular, we verify that there is a $g \in C_p(X)$ such that $(\forall i \leq n)(f(x_i) = g(x_i))$: for each $i \leq n$, let $g_i\in C_p(X)$ be such that $g_i(x_i)=1$ and $g_i(x_j)=0$ for each $i\neq j$. It suffices to take $g(x)=f(x_0)g_0(x)+\ldots +f(x_n)g_n(x)$, which is continuous. Then $\pi_Y(g) \in U$. 

	\item Injectivity follows from the following well-known fact: if $f, g \in C_p(X)$ and they agree on a dense subspace, then $f=g$.

	\item Take a basic open set $W(f; x_0, \ldots, x_n; \epsi)$ of $C_p(X)$. Suppose $x_0, \ldots, x_k \in Y$ and $x_{k+1}, \ldots, x_n \in X\setminus Y$ for some $0 \leq k \leq n$. It suffices to show $\pi_Y[W(f; x_0, \ldots, x_n; \epsi)] = W(\pi_Y(f); x_0, \ldots, x_k; \epsi) \insect \pi_Y\left(C_p(X)\right)$. The inclusion
	$$\pi_Y[W(f; x_0, \ldots, x_n; \epsi)] \subseteq W(\pi_Y(f); x_0, \ldots, x_k; \epsi) \insect \pi_Y\left(C_p(X)\right)$$ 
	is immediate by the definition of $\pi_Y(f)$. To prove the other inclusion, suppose
	\begin{linenomath*} 
	\[g \in W(\pi_Y(f); x_0, \ldots, x_k; \epsi) \insect \pi_Y\left(C_p(X)\right)\]
	\end{linenomath*}
	with $\pi_Y(h) = g$. By complete regularity and the fact that $Y$ is closed, there is a continuous function $l \in C_p(X)$ such that $\rstrct{l}{Y} = 0$ and, for each $k+1 \leq i \leq n$, $l(x_i) = f(x_i) - h(x_i)$. Then $l+h \in W(f; x_0, \ldots, x_n; \epsi)$ and $\pi_Y(l+h) = g$.
\end{enumerate}\vspace{-0.7cm}
\end{proof}

Recall that if $\vphi: X \to Y$ where $X$ and $Y$ are topological spaces, we say that $\vphi$ is a \emph{homeomorphism} if it is a continuous bijection with a continuous inverse. 

A basic notion that will be relevant is the \textit{dual map}: if $f: X \to Y$ is any function, the dual map of $f$ is the function $\Phi_f: \reals^Y \to \reals^X$ given by $\Phi_f(g) = g\com f$. The following result \cite{ArhangelskiiFunction} collects the facts we shall need:

\begin{lem}
\label{lem:dual}
	Let $f: X \to Y$ be any function. Then:
\begin{enumerate}[label=(\roman*)]
	\item $\Phi_f$ is continuous.
	\item \label{thm:dualhomeo}If $f$ is surjective, then $\Phi_f$ is a homeomorphism between $\reals^Y$ and $\Phi_f\left(\reals^Y\right)$.
	\item If $f$ is a continuous surjection, then $\rstrct{\Phi_f}{C_p(Y)}$ is a homeomorphism between $C_p(Y)$ and $\Phi_f\left(C_p(Y)\right) \subseteq C_p(X)$.
\end{enumerate}
\end{lem}

\begin{proof}\mbox{}
\begin{enumerate}[label=(\roman*)]
	\item Suppose $\Phi_f(g) = h$. Consider a basic open neighbourhood 

	$W(h; x_1, \ldots, x_n; \epsi)$. Then
	\begin{linenomath*} 
	\[\Phi_f\left(W(g; f(x_1), \ldots, f(x_n); \epsi)\right) \subseteq W(h; x_1, \ldots, x_n; \epsi).\]
	\end{linenomath*}

	\item Let $g_1, g_2 \in \reals^Y$ be different functions. Take $y \in Y$ such that $g_1(y) \neq g_2(y)$. Since $f$ is surjective, take $x \in f^{-1}(y)$. Then $\Phi_f(g_1)(x) \neq \Phi_f(g_2)(x)$. It remains to show that $\Phi_f^{-1}$ is continuous. Suppose $\Phi_f(g) = h$, and consider a basic open set $W(g; y_1, \ldots, y_n; \epsi)$. By surjectivity, we can take $x_1, \ldots, x_n \in X$ such that $f(x_1) = y_1, \ldots, f(x_n) = y_n$. Then
	\begin{linenomath*} 
	\[\Phi_f^{-1}\left(W(h; x_1, \ldots, x_n; \epsi) \insect \Phi_f\left(\reals^Y\right)\right) \subseteq W(g; y_1, \ldots, y_n; \epsi).\]
	\end{linenomath*}

	\item Immediate from \ref{thm:dualhomeo} by noting that the range $\rstrct{\Phi_f}{C_p(Y)}$ is included in $C_p(X)$.
\end{enumerate}
\end{proof}

The following results sample the interrelations between the topology of $X$ and the one on $C_p(X)$. First, we introduce some topological cardinal functions.

\begin{defi}
	Let $(X, \mcal{T})$ be a topological space.
\begin{enumerate}[label=(\alph*)]
	\item A \emph{base} for $X$ is a family $\mathcal{B}$ of open sets such that $(\forall x \in X)(\forall U \in \mcal{T})[(x \in U) \implies (\exists V \in \mathcal{B})(x \in V \subseteq U)]$. The \emph{weight} $w(X)$ of $X$ is the minimal cardinality of a base for $X$.

	\item A \emph{local base of $X$ at a point $x \in X$} is a family $\mathcal{B}_x$ of open sets such that $x \in \Insect \mathcal{B}_x$ and for any open set $U$ containing $x$, there is some $V \in \mathcal{B}_x$ such that $x \in V \subseteq U$. The \emph{character of $X$ at $x$}, denoted $\chi(x,X)$, is the minimal cardinality of a local base of $X$ at $x$. The \emph{character of $X$} is defined by $\chi(X) = \sup\st{\chi(x,X)}{x \in X}$.
\end{enumerate}
\end{defi}

\begin{thm}
\label{thm:crds}
	Given a topological space $X$, the following equalities hold:
	\begin{linenomath*}
	\[\crd{X} = \chi(C_p(X)) = w(C_p(X)).\]
	\end{linenomath*}
\end{thm}

\begin{proof}
	It follows from the definitions (and our assumption that $X$ is infinite) that
	\begin{linenomath*} 
	\[\chi(C_p(X)) \leq w(C_p(X)) \leq w(\reals^X) \leq \crd{X},\]
	\end{linenomath*}
	so it suffices to prove $\crd{X} \leq \chi(C_p(X))$. Suppose, on the contrary, that the strict inequality $\chi(C_p(X))<\crd{X}$ holds, and denote by $\mathbf{0}$ the function that is identically $0$. Let $\mathcal{B}_0$ be a local base of $C_p(X)$ at $\mathbf{0}$ such that $\crd{B_0} < \crd{X}$. We may assume that $\mathcal{B}_0$ consists of only basic open sets of $C_p(X)$. Now consider the following set
	\begin{linenomath*} 
	\begin{align*}
		Y =
		\st{y \in X}{(\exists n<\omega)(\exists x_0, \ldots, x_n \in &X)(\exists \epsi > 0)\\
		&(W(\mathbf{0}; x_0, \ldots, x_n, y; \epsi) \in \mathcal{B}_0)}.
	\end{align*}
	\end{linenomath*} 
	By our assumption, $\crd{Y} < \crd{X}$. Take $y \in X\setminus Y$ and consider the open set $U = W(\mathbf{0}; y; 1)$, which is a neighbourhood of $\mathbf{0}$. For any $V = W(\mathbf{0}; x_0, \ldots, x_n; \epsi) \in \mathcal{B}_0$, we have $(\forall i < n)(x_i \neq y)$. By complete regularity, there is a $g \in C_p(X)$ such that, for all $i < n$, $g(x_i) = 0$ and $g(y) = 1$. Then $g \in V\setminus U$, contradicting that $\mathcal{B}_0$ is a local base.
\end{proof}

It follows from Theorem \ref{thm:crds} that metrizability and first countability\footnote{A space $X$ is first countable if $\chi(X) \leq \aleph_0$.} are indistinguishable in $C_p(X)$, i.e.~$C_p(X)$ is metrizable if and only if $C_p(X)$ is first countable if and only if $X$ is countable.

On the way toward Grothendieck's theorem, we introduce some definitions:
\begin{defi}\label{ccpt}
Let $X$ be a topological space.
\begin{enumerate}[label=(\alph*)]
	\item A subset $A \subseteq X$ is \emph{relatively compact} (in $X$) if $\clsr{A}$ is compact.

	\item $X$ is \emph{countably compact} if every infinite subset $Y$ of $X$ has a \emph{limit point}, i.e.~a point $x$ such that every open set about $x$ contains infinitely many points of $Y$. If $A \subseteq B \subseteq X$, $A$ is \emph{countably compact in $B$} if every infinite subset of $A$ has a limit point in $B$.

	\item A set of real-valued functions $A \subseteq \reals^X$ is \emph{pointwise bounded} if for every $x \in X$ there is an $M_x>0$ such that for all $f \in A$, $\abs{f(x)} < M_x$.
	
	\item The \emph{\lind number $L(X)$ of $X$} is the smallest infinite cardinal such that every open cover contains a subcover of cardinality $\leq L(X)$. $X$ is \emph{\lind} if $L(X) = \aleph_0$.
\end{enumerate}
\end{defi}

Definition \ref{ccpt} (b) can be restated in terms that justify the choice of words for "countable compactness":

\begin{prop}
    A $T_1$ space\footnote{A space is $T_1$ if each point is a closed set.} $X$ is countably compact if and only if every countable open cover of $X$ includes a finite subcover.
\end{prop}
\begin{proof}
    Suppose $X$ is countably compact and suppose that there is a countable open cover $\{ U_n : n<\omega \}$ that does not include a finite subcover. For each $n<\omega$, pick $x_n\in X\setminus (U_0\cup ... \cup U_{n-1})$, which is possible since $U_0\cup ... \cup U_{n-1}\neq X$. Then $\{x_n :n<\omega\}$ doesn't have a limit point\footnote{A point $x$ is a limit point of a space $Y$ if every open set about $x$ intersects $Y$.}, a contradiction.\par
    Conversely, suppose every countable open cover of $X$ includes a finite subcover and that there is an infinite subset $A\subseteq X$ with no limit points. Then $A$ is closed and discrete and so all its subsets are also closed and discrete. Take a countable subset $B=\{b_n : n<\omega\}\subseteq A$ and let, for each $n<\omega$, $U_n=X\setminus \{b_k : k\geq n\}$. By $T_1$, $\{U_n : n<\omega\}$ is an open cover of $X$ which doesn't include a finite subcover, a contradiction. 
\end{proof}

The following standard result will be used in the proof of Grothendieck's theorem.

\begin{prop}\label{clsrcountablycmpct}
    If $X$ is a normal space and $A\subseteq X$ is countably compact in $X$, then $\clsr{A}$ is countably compact. 
\end{prop}

\begin{proof}
    Suppose that $\clsr{A}$ is not countably compact. Then there is a countably infinite closed discrete subset $B$ of $\clsr{A}$. Let $B = \{b_n\}_{n<\omega}$. It suffices to find disjoint open sets $\{U_n\}_{n < \omega}$, $b_n \in U_n$, for then we can use normality to obtain disjoint open sets $U, W$, $U\supseteq B$, $V\supseteq X\setminus \Union_{n < \omega}U_n$. Then the open sets $\{U \insect U_n\}_{n < \omega}$ form a discrete collection, i.e.~each point in $X$ is at most one of these. Take $a_n \in U\insect U_n \insect A$. Then $\{a_n : n < \omega\}$ has no limit points in $X$, contradicting $A$ being countably compact in $X$. To find the required disjoint open sets, by normality find disjoint open sets $V_n$ containing $b_n$ and $V'_n \supseteq \{b_k : k\neq n\}$. Let $U_n = V_n \insect \Insect\{V'_k : k < n\}$.
\end{proof}

Proposition \ref{clsrcountablycmpct} is not true without some hypothesis such as normality. A standard example is the \emph{tangent disk space} $X$. Consider the unit square, including the $x$-axis, Give the points above the $x$-axis their usual topology. A basic neighbourhood of a point on the axis is the point itself together with an open disk tangent to the axis at that point. Consider the set $Y$ consisting of all points above the $x$-axis with both coordinates rational. Then $Y$ is countably compact in $X$, but its closure is all of $X$, which includes the subset $[0,1]$ of the closed discrete $x$-axis, so is not countably compact.

Now we proceed to prove Grothendieck's Theorem, which also relates properties of $X$ to properties of $C_p(X)$. The proof given here is a variation on the one that appears in \cite{TodorcevicTopics}.

\begin{thm}
	Let $X$ be a countably compact topological space, and $A \subseteq C_p(X)$. Then $A$ is relatively compact (in $C_p(X)$) if and only if it is countably compact in $C_p(X)$.
\end{thm}

\begin{proof}
	If $A$ is relatively compact in $C_p(X)$, then $\clsr{A}\cap C_p(X)$ is compact and so $A$ is countably compact in $C_p(X)$. \par
	
	Conversely, suppose $A$ is countably compact in $C_p(X)$; then it is pointwise bounded (otherwise there would be a sequence $\langle f_n :n<\omega \rangle$ in $A$ and an $x\in X$ such that for each $n<\omega$, $\abs{f_n(x)}>n$, and there can be no limit point of such a sequence). For each $x \in X$, let $M_x > 0$ witness that $A$ is pointwise bounded. Then the closure $\clsr{A}$ in $\reals^X$ is compact since it is a closed subset of the compact $\prod_{x \in X}[-M_x, M_x]$. Then it suffices to show $\clsr{A} \subseteq C_p(X)$.

	Assume to the contrary that there is some $g \in \clsr{A} \setminus C_p(X)$. Then since $g$ is not continuous there are $\epsi > 0$ and $y \in X$ such that $y$ is in the closure of $Y = X\setminus g^{-1}[(g(y)-\epsi, g(y)+\epsi)]$. We recursively build the following sequences: a sequence $\{U_n\}_{n < \omega}$ of open sets containing $y$; a sequence of points $\{x_n\}_{n < \omega}$ in $Y$; and a sequence of functions $\{f_n\}_{n < \omega}$ in $A$ such that the following conditions are satisfied for each $n < \omega$:
	\begin{enumerate}[label=(\roman*)]
		\item \label{cond1} $\clsr{U_{n+1}} \subseteq U_n$.
		\item \label{cond2} $(\forall x \in U_n) (\abs{f_n(x) - f_n(y)} < \epsi/2^n)$.
		\item \label{cond3} $x_n \in U_n \insect Y$.
		\item \label{cond4} $(\forall i < n) (\abs{f_{n+1}(x_i) - g(y)} > \epsi/2)$.
		\item \label{cond5} $\abs{f_n(y) - g(y)} < \epsi/2^n$.
	\end{enumerate}

	First, take $f_0\in W(g;y;\epsi)$, $U_0$ the inverse image under $f_0$ of a ball of radius $\epsi$ around $f_0(y)$ and $x_0\in U_0\cap Y$ (where $U_0\cap Y$ is non-empty since $y$ is a limit point of $Y$).
	
	Suppose $U_i$, $x_i$, and $f_i$ have been constructed for each $i \leq n$. Take
	\begin{linenomath*}  
	\[f_{n+1} \in W(g; x_1, \ldots, x_n; \epsi/2)\insect W(g; y; \epsi/2^{n+1}) \insect A,\]
	\end{linenomath*} 
	 which is nonempty since $g$ is a limit point of $A$. Since we have that $f_{n+1}\in W(g; y; \epsi/2^{n+1})$, it satisfies \ref{cond5}. Note that $f_{n+1}$ also satisfies \ref{cond4} since for each $i<n$,
	 \begin{linenomath*}  
	 \[\abs{f_{n+1}(x_i) - g(y)} \geq \abs{g(x_i) - g(y)} - \abs{f_{n+1}(x_i) - g(x_i)} > \epsi/2\] 
	 \end{linenomath*}
	 by construction. Since $f_{n+1}$ is continuous, we may take an open set $U_{n+1}$ containing $y$ satisfying \ref{cond2}, and such that $\clsr{U_{n+1}} \subseteq U_n$, which is \ref{cond1}; this is done by considering the inverse image under $f_{n+1}$ of a ball of radius at most $\epsi/2^{n+1}$ and centred at $f_{n+1}(y)$, then taking its intersection with $U_n$ and using regularity to shrink this to get $\clsr{U_{n+1}} \subseteq U_n$. We may then pick some $x_{n+1} \in U_{n+1}\insect Y$ so that \ref{cond3} is satisfied.

	 Now $X$ is countably compact, so let $x_\infty$ be a limit point of $\{x_n\}_{n < \omega}$. Then $x_\infty \in \Insect_{n < \omega}\clsr{U_n}$, since $x_{n+1} \in U_{n+1}\subseteq \clsr{U_{n+1}} \subseteq U_n$.

	 Let $S = \{x_n\}_{n < \omega} \union \{x_\infty\}$ and consider the restriction map $\pi_S: C_p(X) \to C_p(S)$ given by $\pi_S(f) = \rstrct{f}{S}$, $f$ restricted to $S$, which is continuous by Lemma \ref{lem:rstrct}. Thus, $F = \pi_S(A)$ is a subset of $\reals^S$, which is metrizable, since $S$ is countable. By the continuity of $\pi_S$, $F$ is countably compact in $C_p(S)$: a countably infinite subset of $F$ is of the form $\{ \rstrct{h_n}{S} : n<\omega \}$ where each $h_n$ belongs to $A$ and so $\{h_n : n<\omega\}$ has a limit point $h\in C_p(X)$. Then $\pi_S(h)=\rstrct{h}{S}\in C_p(S)$ is a limit point of the original set. Since $C_p(S)$ is a metrizable space (by Theorem \ref{thm:crds}), the closure $\clsr{F}$ of $F$ in $C_p(S)$ is countably compact by Proposition \ref{clsrcountablycmpct}. Countable compactness and compactness are equivalent for metrizable spaces (metrizable spaces are \emph{paracompact} and paracompact countably compact spaces are compact, see \cite{Engelking1989}); thus $\clsr{F}$ is compact. In this particular case, there is also a simpler proof: $S$ is countable, so $\reals^S$ has a countable base so is \lind. Clearly, \lind countably compact spaces are compact. Now, let $h \in \clsr{F} \subseteq C_p(S)$ be a limit point of $\{\rstrct{f_n}{S}\}_{n < \omega}$. We verify that our construction implies $h(x_\infty) \notin \clsr{\{h(x_n)\}_{n < \omega}}$.

	 First, notice that \ref{cond2} implies $h(x_\infty) = g(y)$. Assume for a contradiction that there is some $N < \omega$ such that $\abs{h(x_N) - h(x_\infty)} < \epsi/4$, and take $M >N$ such that $\rstrct{f_M}{S} \in W(h; x_N; \epsi/4)$. Thus,
	 \begin{linenomath*} 
	 \[\abs{f_M(x_N) - g(y)} \leq \abs{f_M(x_N) - h(x_N)} + \abs{h(x_N) - h(x_\infty)} < \epsi/2,\] 
	 \end{linenomath*} 
	 which contradicts \ref{cond4}. Thus there is some $r>0$ such that
	 \begin{linenomath*}  
	 \[(\forall n < \omega)(\abs{h(x_n) - h(x_\infty)} > r),\]
	 \end{linenomath*} 
	  and so $h$ is not continuous at $x_\infty$, contradicting $h \in C_p(S)$ and $x_\infty \in S$.

	 We conclude that $\clsr{A} \subseteq C_p(X)$ and so $A$ is relatively compact in $C_p(X)$.
\end{proof}

\begin{defi}
	A \emph{$g$-space} is a topological space $X$ satisfying the following property: for every $A \subseteq X$, if $A$ is countably compact in $X$ then $A$ is relatively compact. We say that $X$ is a \emph{hereditary $g$-space} if every subspace of $X$ is a $g$-space.
\end{defi}

The Fr\'echet-Urysohn property arises naturally in the context of $g$-spaces.

\begin{defi}
	Let $X$ be a topological space. $X$ is \emph{Fr\'echet-Urysohn} if whenever $A \subseteq X$, every point $p$ in $\clsr{A}$ is the limit of a sequence of points in $A$, i.e.~every open set about $p$ contains all but finitely many members of the sequence.
\end{defi}

The following proposition from \cite{Arhangelskii1997a} will be useful for generalizing Grothendieck's Theorem:

\begin{prop}
\label{prop:gspace}
	A $g$-space $X$ is a hereditary $g$-space if and only if every compact subspace of $X$ is Fr\'echet-Urysohn.
\end{prop}

\begin{proof}
	Let $X$ be a $g$-space. First, assume that every compact subspace of $X$ is Fr\'echet-Urysohn. Let $Y \subseteq X$, and let $A \subseteq Y$ be countably compact in $Y$, so it is also countably compact in $X$. Since $X$ is a $g$-space, $\clsr{A}$ is compact. Then $\clsr{A}$ is Fr\'echet-Urysohn by our assumption. Since every element in $\clsr{A}$ is the limit of a sequence in $A$, we have $\clsr{A}\insect(X\setminus Y) = \emptyset$. For if $x\in X\setminus Y$ were a limit of a sequence $\langle x_n : n<\omega \rangle$ from $A$, then $\{ x_n : n<\omega \}$ would contradict $A$ being countably compact in $Y$. Thus $\clsr{A} \subseteq Y$. This shows that $\clsr{A}$ is also compact in $Y$, showing that $Y$ is a $g$-space as desired.

	Conversely, suppose $X$ is a hereditary $g$-space and assume for a contradiction that there is some $A \subseteq X$ such that $\clsr{A}$ is compact and there is an $x \in \clsr{A}$ that is not the limit of a sequence in $A$. Then $Y = \clsr{A}\setminus \{x\}$ is a $g$-space. We claim that $A$ is countably compact in $Y$: take a countably infinite set $\{ x_n : n<\omega \}$ in $A$. Any limit point of such a set must lie in $\overline{A}$. Then it is enough to check that $x$ is not the unique limit point of that set. For the sake of contradiction, suppose that $x$ is the unique limit point of that set. Now consider $\{ x_n : n<\omega \}$ as the sequence $\langle x_n : n<\omega \rangle$. By assumption, no subsequence can converge to $x$. This means that there is a neighbourhood $U$ containing $x$ such that for any $N<\omega$ there are infinitely many $k\geq N$ such that $x_k\notin U$. These $x_k$'s constitute an infinite subset of $A$ with no limit point in the compact $\overline{A}$, which is impossible. Then $x$ cannot be the unique limit point of an infinite subset of $A$ and so $A$ is countably compact in $Y$, so it is relatively compact in the $g$-space $Y$. Since $X$ is a hereditary $g$-space, $A$ is compact, hence $A = \clsr{A}$ is closed. But then $x \in A$, a contradiction.
\end{proof}

The following definition is the starting point for generalizing \groth's Theorem to a broader class of spaces:

\begin{defi}[\cite{Arhangelskii1997a}]
	A topological space $X$ is \emph{Grothendieck} if $C_p(X)$ is a hereditary $g$-space. $X$ is \emph{weakly Grothendieck} if $C_p(X)$ is a $g$-space.
\end{defi}

Any uncountable discrete space is an example of a space that is weakly Grothendieck but not Grothendieck \cite{Arhangelskii1998}.

The property of being a Grothendieck space is transferable from a dense subspace to the full space. This result appears in \cite{Arhangelskii1997a} which quotes Pryce \cite{Pryce1971}. In order to prove it, we need the following lemma, which is a variation on the well-known fact that a continuous bijection between compact spaces is necessarily a homeomorphism.

\begin{lem}
\label{lem:contbijhomeo}
	Let $X$ be countably compact, $Y$ Fr\'echet-Urysohn, and $f: X \to Y$ a continuous bijection. Then $f$ is a homeomorphism.
\end{lem}

\begin{proof}
	Suppose $f: X \to Y$ is a continuous bijection as above, $A \subseteq Y$, and $y \in \clsr{A}$. Since $Y$ is Fr\'echet-Urysohn we have a sequence $S = \{y_n\}_{n < \omega}$ in $A$ converging to $y$. Then no other $x \neq f^{-1}(y)$ can be a limit point of $f^{-1}(S)$. However, by countable compactness, $f^{-1}(S)$ has a limit point, which must therefore be $f^{-1}(y)$. Then $f^{-1}(y) \in \clsr{f^{-1}(A)}$, so $f^{-1}(\clsr{A}) \subseteq \clsr{f^{-1}(A)}$ for any $A \subseteq Y$, hence $f^{-1}$ is continuous and $f$ is a homeomorphism.
\end{proof}

\begin{thm}
\label{thm:densesubspace}
	Let $Y$ be a dense subspace of $X$. If\, $Y$ is Grothendieck, so is $X$.
\end{thm}

\begin{proof}
	Consider the restriction map $\pi_Y: C_p(X) \to C_p(Y)$. By Lemma \ref{lem:rstrct}, $\pi_Y$ is a continuous bijection of $C_p(X)$ onto a subspace $Z$ of $C_p(Y)$. Since $C_p(Y)$ is a hereditary $g$-space, then so is $Z$.

	If $A \subseteq C_p(X)$ is a subspace, and $C$ is countably compact in $A$, then $\pi_Y(C) \subseteq \pi_Y(A)$ is countably compact in $\pi_Y(A)$ and thus relatively compact. By Proposition \ref{prop:gspace}, $\overline{\pi_Y(C)}$ is Fr\'echet-Urysohn and we conclude by using Lemma \ref{lem:contbijhomeo}.
\end{proof}

\arhan \cite{Arhangelskii1997a} proved the preservation of Grothendieck spaces under continuous images:

\begin{thm}
\label{thm:contimage}
	The continuous image of a Grothendieck space is Grothendieck.
\end{thm}

\begin{proof}
	Suppose $X$ is Grothendieck, $Y$ a topological space, and $f: X \to Y$ a continuous surjection. By Lemma \ref{lem:dual}, $C_p(Y)$ is homeomorphic to a subspace of $C_p(X)$. Since $C_p(X)$ is a hereditary $g$-space, so is $C_p(Y)$.
\end{proof}

Notice that the previous two results rely on $C_p(X)$ being a hereditary $g$-space, so the same proof does not work if we merely assume that the space is weakly Grothendieck.

\subsection{Lindel\"of $\Sigma$-spaces are Grothendieck}
\begin{defi}\label{def:pSspaces}
	Let $X$ be a topological space.
	\begin{enumerate}[label=(\alph*)]
		\item $X$ is a \emph{\lind $p$-space} if and only if it can be perfectly mapped onto a space with a countable base. Recall that a map is \emph{perfect} if it is continuous, closed, and the pre-images of points are compact.

		\item $X$ is a \emph{\lind $\Sigma$-space} if and only if it is the continuous image of a \lind $p$-space.
	\end{enumerate}
\end{defi}

These are not the original definitions but rather, equivalent ones. Perfect maps amount to just collapsing some compact sets to points. They are frequently used in general topology since they preserve many properties, either in the direct image or the inverse image. The definitions of $p$-spaces and $\Sigma$-spaces are quite complicated. Notice that countable products of \lind $p$-spaces---in particular, countable spaces---are also \lind $p$-spaces: suppose $\{X_n : n<\omega\}$ is a family of \lind $p$-spaces. For each $n<\omega$, fix a perfect surjection $f_n:X_n\to M_n$ where  $M_n$ is second countable\footnote{A space $X$ is \emph{second countable} if $w(X) \leq \aleph_0$.}. Then the diagonal product $f:\Pi_{n<\omega}X_n\to \Pi_{n<\omega}M_n$ given by $f(x_0,x_1,...)=(f_0(x_0),f_1(x_1),...)$ is perfect (see \cite{Engelking1989}, p.~185, for the details) and its range is a second countable space since second countability is hereditary and preserved by countable products. Thus, $\Pi_{n<\omega}X_n$ is a \lind $p$-space. 

The class of spaces which are \groth has been widely studied by \arhan \cite{Arhangelskii1997a}. Following \cite{Arhangelskii1997a}, we shall prove in the Appendix that
\begin{thm}\label{thm:lindSgroth}
	All \lind $\Sigma$-spaces are \groth.
\end{thm}
\lind $\Sigma$-spaces constitute a nice class as they can be also described as follows: The class of \lind $\Sigma$-spaces is the smallest class of spaces containing all compact spaces, all second countable spaces, and that is closed under finite products, closed subspaces, and continuous images (\cite{Tkachuk2010} contains a detailed exposition of these facts).

Examples of \lind $\Sigma$-spaces include compact spaces, \emph{Polish spaces}, and $K$\textit{-analytic spaces} (see \cite{Tkachuk2010}). They coincide with the $K$\textit{-countably-determined spaces} of \cite{RJ}. We find it useful to think of \lind $\Sigma$-spaces as a weakening of $K$-analytic spaces. Those latter spaces have found applications in Analysis as a generalization of \emph{Polish} (separable completely metrizable) spaces. They can be characterized as continuous images of perfect pre-images of Polish spaces. Thus, \lind $\Sigma$-spaces just drop ``completeness'' from $K$-analytic spaces. These more topological---as contrasted with descriptive set-theoretic---definitions of $K$-analytic and $K$-countably determined spaces were apparently not known to Rogers and Jayne. For further discussion see \cite{Tall2020}.

The proof that all \lind $\Sigma$-spaces are \groth is not easy. We have put it into the Appendix for those who want to learn more $C_p$-theory and are willing to invest some effort. It is put together from several important theorems of $C_p$-theory. It turns out, however, that for the particular application to undefinability of Banach spaces that we are interested in, we only need to prove that \emph{countable} spaces are \groth! The proof of this is quite easy so we give it in the main text. The reason we can get away with this trivial case is that we only deal with countable languages, such as first order continuous logic and countable fragments of continuous $\mcal{L}_{\omega_1, \omega}$.

\begin{defi}
 The \emph{tightness $t(X)$ of $X$} is the smallest infinite cardinal such that whenever $A \subseteq X$ and $x \in \clsr{A}$, there exists a $B \subseteq A$ such that $\crd{B} \leq t(X)$ and $x \in \clsr{B}$. When $t(X)=\aleph_0$, we say that $X$ is \textit{countably tight}.
\end{defi}

Countable tightness plays an important role in the study of \groth spaces, as we shall see in the proof that \lind $\Sigma$-spaces are \groth. There, it is the countable tightness of $C_p(X)$ that is important, and this is obtained from the fact that countable products of \lind $\Sigma$-spaces are \lind. It is also important that

\begin{thm}\label{thm:stightwgroth}
	Countably tight spaces are weakly \groth. 	
\end{thm}

Theorem \ref{thm:stightwgroth} is stated in \cite{Arhangelskii1997a} as part of a much more general result. A different proof is in \cite{TallGrot}.\\

The second author \cite{TallGrot} has recently shown that:

\begin{thm}
    The Proper Forcing Axiom ($\mathbf{PFA}$) implies that \lind countably tight spaces are Grothendieck.
\end{thm}

The conclusion does not hold  if $\mathbf{V=L}$ (see \cite{TallGrot}). Also in the realm of set theory is \arhan's theorem which depends on work of Okunev \cite{Okunev}:
\begin{thm}[\cite{Arhangelskii1997a}]
    $\mathbf{MA_{\aleph_1}}$ implies countably tight spaces with finite powers \lind are Grothendieck.
\end{thm}

As we said, all we will need in order to prove our definability result is:

\begin{thm}
	Every countable space is \groth.
\end{thm}

\begin{proof}
	By Theorem \ref{thm:crds}, if $X$ is countable, $C_p(X)$ has a countable base. It follows easily that all subspaces of $C_p(X)$ are \frech, so we need only show that a space with a countable base is a $g$-space. A space with a countable base is separable metrizable, so normal, so by Proposition \ref{clsrcountablycmpct}, if $A$ is countably compact in $C_p(X)$, then $\clsr{A}$ is countably compact. But $C_p(X)$ is \lind, so $\clsr{A}$ is compact.
\end{proof}

\subsection{Grothendieck Spaces and Double Limit Conditions}
The relationship of \groth's Theorem with double limit conditions was explored by V.~Pták \cite{Ptak1963}:

\begin{thm}
	Let $X$ be compact and $A \subseteq C_p(X)$ pointwise bounded. Then $A$ is relatively compact in $C_p(X)$ if and only if the following double limit condition holds: whenever $\seq{x}{n}{n < \omega}$ is a sequence in $X$ and $\seq{f}{m}{m<\omega}$ is a sequence in $A$, the double limits $\lim_{n \to \infty}\lim_{m \to \infty}f_m(x_n)$ and $\lim_{m \to \infty}\lim_{n \to \infty}f_m(x_n)$ are equal whenever they exist.
\end{thm}

Pták's proof involves a combinatorial result on convex means which we omit, since we are only interested in the analogue of Pták's Theorem for \emph{ultralimits} and without compactness assumptions; see Theorem \ref{thm:doublelimit}. We shall frequently deal with ultralimits in what follows.

\begin{defi}
	Let $X$ be a topological space. Given an ultrafilter $\mcal{U}$ on a regular cardinal $\kappa$, and a $\kappa$-sequence $\seq{x}{\alpha}{\alpha < \kappa}$ in $X$, we say that
	\begin{linenomath*} 
	\[\lim_{\alpha \to \mcal{U}}x_\alpha = x\]
	\end{linenomath*} 
	if and only if for every open neighbourhood $U$ of $x$, $\st{\alpha < \kappa}{x_\alpha \in U} \in \mcal{U}$.
\end{defi}

It follows from the fact that an ultrafilter cannot contain disjoint sets that if an ultralimit exists in a Hausdorff space, it is unique. The following result relates ultralimits and compactness:

\begin{prop}\label{ultracomp}
	A space $X$ is compact if and only if every ultralimit in $X$ exists.
\end{prop}

\begin{proof}
	Suppose $X$ is compact and assume for contradiction that there is an ultrafilter $\mcal{U}$ over $\kappa$ such that $\lim_{\alpha \to \mcal{U}}X_\alpha$ does not exist. Then for each $x \in X$ there is an open neighbourhood $U_x$ of $x$ such that $\st{\alpha < \kappa}{x_\alpha \in U_x} \notin \mcal{U}$. Since $X$ is compact, the open cover $\st{U_x}{x \in X}$ has a finite subcover $\{U_{x_0}, \ldots, U_{x_{n-1}}\}$. Since $\Union_{i < n}U_{x_n} = X$, we must have $\st{\alpha < \kappa}{x_\alpha \in U_{x_i}} \in \mcal{U}$ for at least one $i < n$, a contradiction.

	Conversely, suppose $X$ is not compact. Let $\{U_\alpha\}_{\alpha<\kappa}$ be an open cover that does not have a finite subcover. Let $\kappa$ be the minimal cardinal with that property. Then for each cardinal $\lambda < \kappa$, $\{U_\alpha : \alpha < \lambda\}$ is not a cover, so we can pick an $x_\lambda \notin \Union\{U_\alpha : \alpha < \lambda\}$. Then each $x \in X$ has a neighbourhood $N_x$ that contains fewer than $\kappa$ many $x_\lambda$'s. Then $\left\{\{\alpha : x_\alpha\notin N_x\} : x \in X\right\}$ forms a filter on $\kappa$, which can be extended to an ultrafilter $\mcal{U}$. Then $\mcal{U}$ has no ultralimit, for given any $x \in X$, $\{\alpha : x_\alpha \notin N_x\} \in \mcal{U}$, so $\{\alpha : x_\alpha \in N_x\} \notin \mcal{U}$.
\end{proof}

In the following sections we shall present generalizations of the main results of \cite{Casazza} and \cite{CDI} by working with \groth spaces instead of with compact spaces. For this purpose we introduce some definitions which are different from the ones in \cite{Casazza} but identical with the added assumption of compactness.

Thorough introductions to continuous logic can be found in \cite{Eagle2014}, \cite{Eagle2015} and \cite{BenYaacov2008}. In \cite{Hamel}, the authors introduced in detail a setting for continuous logic, mostly following \cite{Casazza} and \cite{Eagle2015}.

\section{Stability, Definability, and Double (Ultra)limit Conditions}
It is worth noting that there are different definitions of continuous logics in the literature. However, most authors differ only slightly from one another. We shall now introduce metric structures following Eagle \cite{Eagle2014} and \cite{Eagle2015} as this is the context we work in:

\begin{defi}\mbox{}
\begin{enumerate}[label=(\alph*)]
	\item If $(M,d)$ and $(N,\rho)$ are metric spaces and $f\colon M^n\to N$ is uniformly continuous, a \textit{modulus of uniform continuity of} $f$ is a function 
    $\delta \colon (0,1)\cap{\mathbb{Q}} \to (0,1)\cap{\mathbb{Q}}$ such that whenever $\boldsymbol{a}=(a_1,...,a_n),\ \boldsymbol{b}=(b_1,...,b_n)\in M^n$ and $\varepsilon \in (0,1)\ \cap\ {\mathbb{Q}}$, $\sup\{d(a_i,b_i) : 1\leq i \leq n\}<\delta(\varepsilon)$ implies
    $\rho(f(\boldsymbol{a}),f(\boldsymbol{b}))<\varepsilon$.

	\item A \emph{language} for metric structures is a set $L$ which consists of constants, functions with an associated arity, and a modulus of uniform continuity; predicates with an associated arity and a modulus of uniform continuity; and a symbol $d$ for a metric.

	\item An \emph{$L$-metric structure} $\mfrak{M}$ is a metric space $(M, d^M)$ together with interpretations for each symbol in $L$: $c^{\mfrak{M}} \in M$ for each constant $c \in L$; $f^{\mfrak{M}}: M^n \to M$ a uniformly continuous function for each $n$-ary function symbol $f \in L$; $P^{\mfrak{M}}: M^n \to [0, 1]$ a uniformly continuous function for each $n$-ary predicate symbol $P \in L$. Assume for the sake of notation simplicity that all metric structures have diameter $1$. 

	\item The space of $L$-structures in a given language $L$ is the family $\mathrm{Str}(L)$ of all $L$-metric structures endowed with the topology generated by the following basic closed sets: $[\vphi] = \st{\mfrak{M}}{\mfrak{M} \ent \vphi}$ where $\vphi$ is an $L$-sentence.
\end{enumerate}
\end{defi}

\begin{rmk*}
	The continuous first-order formulas are constructed just as in the discrete setting except for the following addition: if $f: [0, 1]^n \to [0, 1]$ is a continuous function and $\vphi_0, \ldots, \vphi_{n-1}$ are $L$-formulas, then $f(\vphi_0, \ldots, \vphi_{n-1})$ is also an $L$-formula. It is customary to identify formulas $\vphi$ of arity $n$ with functions $\mf{M}^n \to [0, 1]$. This allows an easy way to define the satisfaction relation, i.e.~if $\vphi$ is an $L$-formula, $\mf{M}$ an $L$-structure, and $a \in \mf{M}^n$, then $\mf{M} \ent \vphi(a)$ if and only if $\vphi(a) = 1$. For example, under this identification, the conjunction, $\vphi \lnd \psi$, of two $L$-formulas $\vphi$ and $\psi$ is $\max(\vphi, \psi)$. In what follows, we shall introduce other continuous logics and study Gowers' problem in them; our main example is continuous $\mcal{L}_{\omega_1, \omega}$ as introduced by Eagle \cite{Eagle2014}, \cite{Eagle2015}. See \cite{Eagle2014}, \cite{Eagle2015} and \cite{Hamel} for a detailed exposition.
\end{rmk*}

A key technical concept introduced in \cite{Casazza} and used to talk about a pair of Banach space norms is the following:

\begin{defi}\label{def:pairlang}\mbox{}
\begin{enumerate}[label=(\alph*)]
	\item Let $L$ be a language $L' \supseteq L$ is a \emph{language for pairs of structures from $L$}, if $L'$ includes two disjoint copies of $L$ and there is a map $\mr{Str}(L) \times \mr{Str}(L) \to \mr{Str}(L')$ which assigns to every pair of $L$-structures $(\mf{M}, \mf{N})$ an $L'$-structure $\langle\mf{M}, \mf{N} \rangle$.

	\item Let $L'$ be a language for pairs of structures from $L$, and $X, Y$ function symbols from $L$. We say that a formula $\vphi(X, Y)$ is a \emph{formula for pairs of structures from $L$} if
	\begin{linenomath*}  
	\[(\mf{M}, \mf{N}) \mapsto \vphi(X^{\mf{M}}, Y^{\mf{N}}) = \mr{Val}\left(\vphi(X, Y), \langle\mf{M}, \mf{N} \rangle\right)\]
	\end{linenomath*} 
	is separately continuous on $\mr{Str}(L) \times \mr{Str}(L)$. For simplicity, we write $\vphi(\mf{M}, \mf{N})$ instead of $\mr{Val}\left(\vphi(X, Y), \langle\mf{M}, \mf{N} \rangle\right)$.
\end{enumerate}
\end{defi}

In this section we shall only be interested in $\vphi$-types when $\vphi$ is a formula for pairs of structures:

\begin{defi}
	Let $L$ be a language, $\vphi$ an $L$-formula for pairs of structures, and $\mf{M} \in \mr{Str}(L)$. The \emph{left $\vphi$-type of $\mf{M}$} is the function $\mr{ltp}_{\vphi, \mf{M}}: \mr{Str}(L) \to [0, 1]$ given by $\mr{ltp}_{\vphi, \mf{M}}(\mf{N}) = \vphi(\mf{M}, \mf{N})$.

	The \emph{space of left $\vphi$-types}, denoted $S^{l}_{\vphi}$ is the closure of $\st{\mr{ltp}_{\vphi, \mf{M}}}{\mf{M} \in \mr{Str}(L)}$ in $C_p(\mr{Str}(L))$. If $C$ is a subset of $\mr{Str}(L)$, then $S^{l}_{\vphi}(C)$ is the closure of $\st{\rstrct{\mr{ltp}_{\vphi, \mf{M}}}{C}}{\mf{M} \in C}$ in $C_p(C)$ and it is called the \emph{space of left $\vphi$-types over $C$}.

	The definitions for \emph{right $\vphi$-types} are analogous.
\end{defi}

\begin{rmk*}
	Notice that the previous definition coincides with the one given in \cite{Casazza} for the compact case: if the logic is compact, that is to say $\mr{Str}(L)$ is compact, then $S^l_{\vphi}$ turns out to be the closure of $\st{\mr{ltp}_{\vphi, \mf{M}}}{\mf{M} \in \mr{Str}(L)}$ in $[0, 1]^{\mr{Str}(L)}$, which is then also compact. Since compact spaces are \groth, we have that $S^l_{\vphi}$ is always Fr\'echet-Urysohn when the logic is compact---see Proposition \ref{prop:gspace}.
\end{rmk*}

The notion of model-theoretic stability has been extensively studied in discrete model theory, whereas works like \cite{Casazza}, \cite{CDI} or \cite{Hamel} are devoted to understanding stability in the continuous setting. Among the multiple equivalent definitions of stability, we have found the one introduced in \cite{Iovino1999} involving a \emph{double limit condition} is the most suitable for the continuous setting, even when it comes to non-compact logics:

\begin{defi}
	Let $L$ be a language for pairs of structures, $\vphi$ a formula for pairs of structures, and $C \subseteq \mr{Str}(L)$. We say that $\vphi$ is \emph{stable on $C$} if and only if whenever $\seq{\mf{M}}{i}{i<\omega}$ and $\seq{\mf{N}}{j}{j<\omega}$ are sequences in $C$, and $\mcal{U}$ and $\mcal{V}$ are ultrafilters on $\omega$, we have
	\begin{linenomath*} 
	\[\lim_{i \to \mcal{U}}\lim_{j \to \mcal{V}}\vphi(\mf{M}_i, \mf{N}_j) = \lim_{j \to \mcal{V}}\lim_{i \to \mcal{U}}\vphi(\mf{M}_i, \mf{N}_j).\]
	\end{linenomath*} 
\end{defi}

\begin{rmk*}
Notice that $\vphi(\mf{M}_i, \mf{N}_j)$ is always a number in $[0,1]$ and so the ultralimits always exist by Proposition \ref{ultracomp}.
\end{rmk*}

We now state the topological version of one of our main results, which varies the one that appears in \cite{Hamel} by replacing ``weakly Grothendieck'' by ``Grothendieck'' and ``first countable'' by ``countably tight'':

\begin{thm}
\label{thm:doublelimit}
	Let $X$ be \groth and countably tight. A subset $A$ of $C_p(X, [0, 1])$ is relatively compact in $C_p(X, [0, 1])$ if and only if it satisfies the following double limit condition: for every pair of sequences $\seq{f}{n}{n < \omega} \subseteq A$ and $\seq{x}{m}{m < \omega} \subseteq X$, and ultrafilters $\mcal{U}$ and $\mcal{V}$ on $\omega$, the double limits
	\begin{linenomath*} 
	\[\lim_{n \to \mcal{U}}\lim_{m \to \mcal{V}}f_n(x_m) = \lim_{m \to \mcal{V}}\lim_{n \to \mcal{U}}f_n(x_m)\]
	\end{linenomath*} 
	agree whenever $\lim_{m \to \mcal{V}}x_m$ exists.
\end{thm}

\begin{proof}
	Let $\clsr{A}$ denote the closure of $A$ in $[0, 1]^X$. Suppose $\clsr{A} \insect C_p(X)$ is not compact in $C_p(X)$. Then $\clsr{A} \insect C_p(X)$ is closed but it is not countably compact in $C_p(X)$, since $X$ is \groth. Let $\seq{f}{n}{n < \omega}$ be a sequence in $A$ with no limit points in $\clsr{A} \insect C_p(X)$. Since $\clsr{A}$ is a compact subset of $[0, 1]^X$, each ultralimit of the sequence exists, is a limit point, and is therefore discontinuous. Take a non-principal ultrafilter $\mcal{U}$ over $\omega$ and let $\lim_{n \to \mcal{U}}f_n = g$, where $g$ is discontinuous by assumption. Then there are $\epsi>0$ and $y \in X$ such that $y \in \clsr{Y}$, where
	\begin{linenomath*}  
	\[Y = X \setminus g^{-1}\left(g(y)-\epsi, g(y)+\epsi\right).\]
	\end{linenomath*} 
	Since $t(X) = \aleph_0$, there is some $Z \subseteq Y$ with $\crd{Z} = \aleph_0$ and $y \in \clsr{Z}$. Suppose $Z = \seq{z}{n}{n < \omega}$ and for each open neighbourhood $U$ of $y$, let $M_U = \st{n < \omega}{z_n \in U}$. Clearly, the family of all $M_U$ is centred and so it can be extended to an ultrafilter $\mcal{V}$ on $\omega$ so that
	\begin{linenomath*} 
	\[\lim_{n \to \mcal{U}}\lim_{m \to \mcal{V}}f_n(x_m) = g(y)\]
	\end{linenomath*} 
	since each $f_n$ is continuous. On the other hand,
	\begin{linenomath*} 
	\[\lim_{m \to \mcal{V}}\lim_{n \to \mcal{U}}f_n(x_m) = \lim_{m \to \mcal{V}}g(x_m)\]
	\end{linenomath*} 
	exists by compactness of $[0, 1]$. However, by the choice of each $x_m$, we have $\abs{g(y) - g(x_m)} > \epsi$ and so the ultralimits exist but are different, a contradiction.

	Conversely, suppose $\clsr{A} \insect C_p(X)$ is compact and $\lim_{m \to \mcal{V}}x_m = y$. Then for any sequence $\seq{f}{n}{n < \omega} \subseteq A$ and ultrafilter $\mcal{U}$ on $\omega$, there is a continuous function $g = \lim_{n \to \mcal{U}}f_n$. Thus,
	\begin{linenomath*} 
	\[\lim_{n \to \mcal{U}}\lim_{m \to \mcal{V}}f_n(x_m) = \lim_{n \to \mcal{U}}f_n(y) g(y),\]
	\end{linenomath*} 
	and
	\begin{linenomath*} 
	\[\lim_{m \to \mcal{V}}\lim_{n \to \mcal{U}}f_n(x_m) = \lim_{m \to \mcal{V}}g(x_m) = g(y).\]
	\end{linenomath*} 
\end{proof}

The main application we consider here deals with countable fragments of $\mcal{L}_{\omega_1, \omega}$, for the corresponding space of types is Polish \cite{Morley1974}, so \groth and countably tight.

The model-theoretic version of Theorem \ref{thm:doublelimit} will relate definability to a double limit condition just as in the discrete compact case. For this purpose, we introduce the notion of definability for $\vphi$-types, where $\vphi$ is a formula for pairs of structures, following \cite{Casazza}:

\begin{defi}
	Suppose $L$ is a language for pairs of structures and $\vphi$ is a formula for pairs of structures. Let $C \subseteq \mr{Str}(L)$. A function $\tau: S^{r}_{\vphi}(C) \to [0, 1]$ is a \emph{left global $\vphi$-type over $C$} if there is a sequence $\seq{\mf{M}}{\alpha}{\alpha < \kappa} \subseteq C$, and an ultrafilter $\mcal{U}$ on $\kappa$, such that for every type $t \in S^{r}_{\vphi}(C)$, say $t = \lim_{\beta \to \mcal{V}}\mr{rtp}_{\vphi,\mf{M}}$, we have
	\begin{linenomath*} 
	\[\tau(t) = \lim_{\alpha \to \mcal{U}}\lim_{\beta \to \mcal{V}}\vphi(\mf{M}_\alpha, \mf{N}_\beta).\]
	\end{linenomath*} 
	We say that $\tau$ is \emph{explicitly definable} if it is continuous. If a left $\vphi$-type is given by $p = \lim_{\alpha \to \mcal{U}}\mr{ltp}_{\vphi,\mf{M}_\alpha}$, we say that $p$ is \emph{explicitly definable} if the respective $\tau$ is continuous.
\end{defi}

The previous definition is sound since in continuous logics the allowable formulas are continuous functions. Moreover, as with most concepts in continuous logic, such a definition reduces to the usual notion in the discrete setting. Also notice that, although our main application deals with Banach spaces, the context of the previous definition and the results in this section is that of the languages for pairs of structures rather than just Banach spaces. We proceed with the model-theoretic version of Theorem \ref{thm:doublelimit}:

\begin{thm}\label{maindlc}
	Let $L$ be a language for pairs of structures, and $\vphi$ a formula for pairs of structures. Suppose $C \subseteq \mr{Str}(L)$ is such that $S^{l}_{\vphi}(C)$ is \groth and countably tight. Then the following are equivalent:
	\begin{enumerate}[label=(\roman*)]
		\item \label{equi1} Whenever a left type over $C$ is given by $t = \lim_{i \to \mcal{U}}\mr{ltp}_{\vphi, \mf{M}_i}$, and $\seq{\mf{N}}{j}{j<\omega} \subseteq C$ is a sequence in $C$, and $\mcal{V}$ is an ultrafilter on $\omega$, then
		\begin{linenomath*} 
		\[\lim_{i \to \mcal{U}}\lim_{j \to \mcal{V}}\vphi(\mf{M}_i, \mf{N}_j) = \lim_{j \to \mcal{V}}\lim_{i \to \mcal{U}}\vphi(\mf{M}_i, \mf{N}_j).\]
		\end{linenomath*} 

		\item \label{equi2} Whenever a left type over $C$ is given by $t = \lim_{i \to \mcal{U}}\mr{ltp}_{\vphi, \mf{M}_i}$, and $\seq{\mf{N}}{j}{j<\omega} \subseteq C$ is a sequence in $C$, then
		\begin{linenomath*} 
		\[\sup_{i < j}\vphi(\mf{M}_i, \mf{N}_j) = \inf_{j < i}\vphi(\mf{M}_i, \mf{N}_j).\]
		\end{linenomath*} 

		\item \label{equi3} If $\tau$ is a left $\vphi$-type over $C$, then $\tau$ is explicitly definable.
	\end{enumerate}
\end{thm}

\begin{proof}
	For a proof of the equivalence of \ref{equi1} and \ref{equi2}, see \cite{Iovino1999} or \cite{Hamel}. To see that \ref{equi1} implies \ref{equi3}, suppose $\langle \mf{N}_{\alpha} : \alpha<\kappa\rangle$ is a $\kappa$-sequence in $C$ such that 
	\begin{linenomath*} 
	\[\tau(\lim_{i \to \mcal{U}}\mr{ltp}_{\vphi, \mf{M}_i}) = \lim_{i \to \mcal{U}}\lim_{\alpha \to \mcal{V}}\vphi(\mf{M}_i, \mf{N}_{\alpha})\]
	\end{linenomath*} 
	defines a global $\vphi$-type, where $\mcal{V}$ is an ultrafilter on $\kappa$. Define $f_{\alpha}: S^l_{\vphi}(C) \to [0, 1]$ by
	\begin{linenomath*} 
	\[f_{\alpha}(\lim_{i \to \mcal{U}}\mr{ltp}_{\vphi, \mf{M}_i}) = \lim_{i \to \mcal{U}}\vphi(\mf{M}_i, \mf{N}_{\alpha}).\]
	\end{linenomath*} 
	Then the set $A = \{f_{\alpha} : \alpha<\kappa\}$ is a subset $A \subseteq C_p(S^l_{\vphi}(C), [0, 1])$ and satisfies the double limit condition, so that $\clsr{A} \insect C_p(S^l_{\vphi}(C), [0, 1])$ is compact and $\tau = \lim_{\alpha \to \mcal{V}}f_{\alpha}$ is continuous. That \ref{equi3} implies \ref{equi1} is immediate.
\end{proof}

One often deals with the space of $\vphi$-types over a subspace $C$ of $\mr{Str}(L)$, $S^{l}_\vphi(C)$, which by Lemma \ref{lem:rstrct} is a continuous image of $S^{l}_\vphi$. An important application of the $C_p$-theoretic machinery we have developed deals with the preservation of the \groth property from $S^l_\vphi$ to $S^l_\vphi(C)$ for any $C \subseteq \mr{Str}(L)$:

\begin{thm}
\label{thm:preserve}
	Let $L$ be a language for pairs of structures, and $\vphi$ a formula for pairs of structures. If $S^l_\vphi$ is \groth and $C \subseteq \mr{Str}(L)$, then $S^l_\vphi(C)$ is \groth.
\end{thm}

\begin{proof}
	Let $A = \st{\mr{ltp}_{\vphi,\mf{M}}}{\mf{M} \in C} \subseteq S^l_\vphi$. By Lemma \ref{lem:rstrct}, the restriction map $\pi_C: C_p(\mr{Str}(L)) \to C_p(C)$ is continuous. Notice that $\pi_C(A)$ is \groth by Theorem \ref{thm:contimage}. By definition, $S^l_\vphi(C) = \clsr{\pi_C(A)}$ and so it is \groth by Theorem \ref{thm:densesubspace}.
\end{proof}

The following result is routine:

\begin{thm}
	The continuous open image of a Fr\'echet-Urysohn space is Fr\'echet-Urysohn.
\end{thm}

\begin{proof}
	Let $X$ be Fr\'echet-Urysohn, and $f: X \to Y$ a continuous open surjection. Let $B \subseteq Y$ and take $y \in \clsr{B}$. By continuity it is always true that $f^{-1}(\clsr{B}) \supseteq \clsr{f^{-1}(B)}$. To prove the reverse inclusion, take $x \in f^{-1}(\clsr{B})$ and an open neighbourhood $U$ of $x$. Since $f$ is open, $f(U)$ is an open neighbourhood of $f(x) \in \clsr{B}$ and so $f(U) \insect B \neq \emptyset$. Then $U \insect f^{-1}(B) \neq \emptyset$, and so $x \in \clsr{f^{-1}(B)}$.

	Next, suppose $f(x_0) = y$. Since $X$ is Fr\'echet-Urysohn and $x_0 \in \clsr{f^{-1}(B)}$, we can take a sequence in $f^{-1}(B)$ converging to $x$, say $\seq{z}{n}{n<\omega}$. By continuity, $\lim_{n \to \infty}f(z_n) = f(x_0) = y$, where $f(z_n) \in B$ for each $n < \omega$.
\end{proof} 

Using Theorem \ref{thm:preserve}, we get:

\begin{crl}
\label{crl:preserve}
	Let $L$ be a language for pairs of structures, and $\vphi$ a formula for pairs of structures. If $S^l_\vphi$ is Fr\'echet-Urysohn and $C \subseteq \mr{Str}(L)$ is closed, then $S^l_\vphi(C)$ includes a dense Fr\'echet-Urysohn subspace.
\end{crl}

\begin{proof}
Notice that $S^l_\vphi(C)$ is the closure of the image under the restriction map $pi_C$ of $S^l_\vphi$. By Lemma \ref{lem:rstrct}, $\pi_C$ is continuous and open. We conclude by Theorem \ref{thm:preserve}
\end{proof}

\begin{rmk*}
The importance of compactness in Model Theory is well understood: the pioneers of the discipline built around this topological concept some of the fundamental pillars. Even though some famous open problems on compactness remain open, topologists have reached a deep understanding of the subject, and model theorists have also extensively explored the applications of compactness in first-order logic and its variations. In the last decades, the work of S.~Shelah has brought to the foreground impressive model-theoretic results that rely on Combinatorics and Set Theory, making the use of these resources the fashionable approach in today's Model Theory.   

Reasonable changes to first-order logic make the statement of the Compactness Theorem false --- consider for instance $L_{\omega_1,\omega}$. This does not completely disable the combinatorial tools but definitely sets important limitations. Recent work \cite{Hamel2021} has shown that the intersection of Topology and Model Theory goes way beyond compactness. As opposed to the relation between compactness and first-order Model Theory, the connections between Topology and Model Theory beyond the Compactness Theorem have barely been explored. It is an interesting question whether it is reasonable to hope for a dictionary that translates model-theoretic terms into topological ones. 

The authors have already shown how more advanced Topology can shine a light on modern model-theoretic results and questions on \emph{metastability} and have shown how to convert topological spaces into abstract logics \cite{Hamel2021}. As mentioned before, this present work was inspired by the work of P.~Casazza and J.~Iovino \cite{Casazza} because of their use of Grothendieck's theorem in the context of compactness. Recently, E.~Due\~nez joined Casazza and Iovino, and \cite{Casazza} was reworked into \cite{CDI}. 

There are clear advantages in the new approach of \cite{CDI} but they come with a price tag for those who are interested in the underlying topology. Even though the results are essentially the same, \cite{CDI} resorts to more advanced combinatorial results involving metastability to yield their main proofs without the need of talking about languages for pairs of structures, which instead were omnipresent in \cite{Casazza}. In \cite{CDI}, the topology that we discuss here and that we use to produce generalizations hides away behind these combinatorial results which ultimately rely on compactness, limiting the applicability of their methods to other, non-compact logics. 

The issue at hand is the following: given a language $L$, an $L$-formula $\varphi(x,y)$, an $L$-structure $\mathcal{M}$ and $a, b$ in the universe of $\mathcal{M}$, it is customary to associate $\varphi^\mathcal{M}(a, b)$ to be the truth value of $\varphi(a, b)$ in $\mathcal{M}$. The formulas are evaluated in elements of $\mathcal{M}$ where usually $\mathcal{M}$ carries little to no topology. This is the approach of \cite{CDI}. On the other hand, working with a language for pairs of structures $L$ turns the relevant formulas into continuous functions from $\mathrm{Str}(L)\times \mathrm{Str}(L)$ into $[0,1]$.  $\mathrm{Str}(L)$, just as the space of types that appears further below, can carry quite interesting topologies; this is what we work with and what the original approach from \cite{Casazza} used. The final goal is to be able to combine two normed spaces into an enlarged structure with two norms where we can compare them --- this is Definition \ref{def:pairlang}. 

Notice that an approach in which we start with a language with two function symbols $\norm{}_1, \norm{}_2$ for norms also reduces the maneuverability that the language for pairs of structures offer. One could certainly obtain continuous functions from $\mathrm{Str}(L)$ to $[0,1]$ by assigning to $\mathcal{M}$ the truth value of  $\varphi^\mathcal{M}(\norm{}_1^\mathcal{M}, \norm{}_2^\mathcal{M})$. This approach is an oversimplification since the applicability of the double limit condition becomes unclear once you work with only one variable. There is a similarity here with the use of bipartite formulas in classical model theory: $\varphi(\bar{x},\bar{y})$, where $\bar{x},\bar{y}$ are tuples, instead of simply working with $\varphi(\bar{z})$ where $\bar{z}$ is a longer tuple.

The authors realize that languages for pairs of structures can be cumbersome. However, they are a key component in the topological presentation of \cite{Casazza} and the corresponding generalization that appears in this work. As a result, we have kept the original approach of \cite{Casazza} here. This approach has its own limitations as it makes it more difficult to use similar techniques to talk about other types of languages, formulas or theories. 
\end{rmk*}

\section{Applications and Examples Concerning the Undefinability of Pathological Banach Spaces}
We first introduce some concepts from Analysis which we will deal with in what follows. We denote the Banach space of sequences of real numbers that are eventually $0$ by $c_{00}$. The space of sequences of real numbers converging to $0$ is denoted by $c_0$. Finally, $\ell^p$ denotes the space of sequences $\seq{x}{n}{n<\omega}$ of real numbers such that $\sum_{n < \omega}\abs{x_n}^p < \infty$.

B.~Tsirelson \cite{Tsirelson1974} constructed a Banach space which does not include a copy of any $\ell^p$ or $c_0$. What is now called \textit{Tsirelson's space} is due to T.~Figiel and W.~Johnson \cite{Figiel1974}: let $\seq{t}{n}{n<\omega}$ be the canonical basis for $c_{00}$. If $x = \sum_{n<\omega}a_nx_n \in c_{00}$ and $E, F \in [\omega]^{<\omega}$, we denote $\sum_{n \in E}a_nx_n$ by $Ex$, and write $E \leq F$ if and only if $\max E \leq \min F$.

Recursively define
\begin{linenomath*} 
\begin{align*}
 	&\norm{x}_0 = \norm{x}_{c_{00}}, \\
 	&\norm{x}_{n+1} = \max\left\{\norm{x}_n, \frac{1}{2}\max_k\left\{\sum_{i<k}\norm{E_i}_n\,:\,\{k\} \leq E_1 < \cdots < E_k\right\}\right\}.
 \end{align*}
 \end{linenomath*} 
 Then,
 \begin{linenomath*} 
 \[\norm{x}_n \leq \norm{x}_{n+1} \leq \norm{x}_{\ell^1}\]
 \end{linenomath*} 
 and we denote $\norm{x}_T = \lim_{n \to \infty}\norm{x}_n$. Then,
 \begin{linenomath*} 
 \[\norm{x}_T = \max\left\{\norm{x}_{c_{00}}, \frac{1}{2}\max\left\{\sum_{i<k}\norm{E_ix}_T\,:\,\{k\} \leq E_1 < \cdots < E_k\right\}\right\}\]
 \end{linenomath*} 
 and $T$ is the norm completion of $\norm{}_T$. The reader is referred to P.~Casazza and T.~Shura \cite{Casazza1989} for a detailed exposition of Tsirelson's space. 

 The context of the subsequent results is what is now known as \textit{Gowers' problem}: does every infinite dimensional explicitly definable Banach space include an isomorphic copy of some $\ell^p$ or $c_0$? We found in \cite{Casazza} the inspiration for the framework and results presented in the current paper. The topological assumption used in \cite{Casazza} is countable compactness whereas ours is \groth plus countable tightness. 
 
  There are mistakes in \cite{Casazza} (which we avoid here) which are corrected in the rewrite \cite{CDI}, which lead the authors to now only assert their results for compact logics. With the replacement of countable compactness by compactness, \cite{Casazza} can be corrected, but with this assumption, the approach of \cite{CDI} is more elegant.
 
 Countable compactness does imply \groth, but not countable tightness. In fact, even compactness does not imply countable tightness. In terms of results, we can trivially deduce their main definability results from ours because finitary continuous logic is a countable fragment (see below) of infinitary continuous logic. However the more general definability results for compact logics are obtained using a double limit condition on sequences obtained by Pták's combinatorial approach.

 To obtain our results, we needed to instead use a double limit condition involving ultralimits and assume countable tightness. In work in progress, we obtain a common generalization of “compactness” and “Grothendieck plus countable tightness” that yields the desired undefinability results. 

 Fix a language $L$ for a pair of structures. Let $C$ be the set of all $L$-structures which are normed spaces based on $c_{00}$ (see definition below). We introduce the continuous logic formula for pairs of structures used in \cite{Casazza}: for norms $\norm{}_1$ and $\norm{}_2$, define
 \begin{linenomath*} 
 \[D(\norm{}_1, \norm{}_2) = \sup\left\{\frac{\norm{x}_1}{\norm{x}_2}\,:\,\norm{x}_{\ell^1} = 1\right\}.\]
 \end{linenomath*} 
 Then define
 \begin{linenomath*} 
 \begin{equation}
 \tag{$\star$}
 \label{def:vphi}
 \vphi(\norm{}_1, \norm{}_2) = 1 - \frac{\log D(\norm{}_1, \norm{}_2)}{1+\log D(\norm{}_1, \norm{}_2)}.
 \end{equation}
 \end{linenomath*} 

 It is not difficult to see that if $t = \lim_{i \to \mcal{U}}\mr{ltp}_{\vphi, \norm{}_i}$ is realized by a structure $(\mf{M}, \norm{}_*)$, then
 \begin{linenomath*} 
 \begin{equation}
 \label{eq:limnorm}
 	\lim_{i \to \mcal{U}}\sup_{\norm{x}_{\ell^1} = 1}\frac{\norm{x}_i}{\norm{x}} = \sup_{\norm{x}_{\ell^1} = 1}\frac{\norm{x}_*}{\norm{x}}
 \end{equation}
 \end{linenomath*} 
 for any structure $(c_{00}, \norm{x}) \in C$. In particular, taking $\norm{}_{\ell^1} = \norm{}$ yields $\lim_{i \to \mcal{U}}\norm{x}_i = \norm{x}_*$ for each $x \in c_{00}$. If in addition 
 \begin{linenomath*} 
 \[\norm{}_1 \leq \norm{}_2 \leq \cdots \leq \norm{}_n \leq \cdots\]
 \end{linenomath*} 
 then $\lim_{i \to \infty}\norm{x}_i = \norm{x}_*$.

 Conversely, if
 \begin{linenomath*} 
 \[\norm{}_1 \leq \norm{}_2 \leq \cdots \leq \norm{}_n \leq \cdots\]
 \end{linenomath*} 
 and also $\lim_{i \to \infty}\norm{x}_i = \norm{x}_*$ for every $x \in c_{00}$, then \eqref{eq:limnorm} holds for any $(c_{00}, \norm{}) \in C$ and any nonprincipal ultrafilter $\mcal{U}$ over $\omega$.

 This motivates the following definition from \cite{Casazza}:

 \begin{defi}
 \label{def:uniquelydetermined}
 	Let $\mcal{C}$ be the set of all structures which are normed spaces based on $c_{00}$, and $\vphi$ a formula for a pair of structures, and $\norm{}_*$ a norm on $c_{00}$.
 	\begin{enumerate}[label=(\alph*)]
 		\item If $\st{\norm{}_i}{i < \omega}$ is a family of norms on $c_{00}$ we say that $\st{\mr{ltp}_{\vphi, \norm{}_i}}{i < \omega}$ \emph{determines $\norm{}_*$ uniquely} if, for every $\mcal{U} \in \beta\omega$, the type $t = \lim_{i \to \mcal{U}}\mr{ltp}_{\vphi, \norm{}_i}$ is realized, and $\norm{}_*$ is its unique realization.

 		\item We say that $\norm{}_*$ is \emph{uniquely determined by its $\vphi$-type over $\mcal{C}$} if there is a family of norms $\st{\norm{}_i}{i < \omega}$ on $c_{00}$ such that $\st{\mr{ltp}_{\vphi, \norm{}_i}}{i < \omega}$ determines $\norm{}_*$ uniquely.
 	\end{enumerate}
 \end{defi}

 Notice that if $\norm{}_i$ denotes the $i$-th iterate of the Tsirelson norm, and $\norm{}_T$ the Tsirelson norm, then $\lim_{i \to \infty}\norm{x}_i = \norm{x}_T$ for each $x \in c_{00}$, and so we have the following result from \cite{Casazza} that follows from the previous definition:

 \begin{prop}
 	Let $L$ be a language for pairs of structures, $\mcal{C}$ the class of structures $(c_{00}, \norm{}_{\ell^1}, \norm{})$ such that the norm completion of $(c_{00}, \norm{})$ is a Banach space including $\ell^p$ or $c_{00}$, and let $\vphi(X, Y)$ be the formula defined by \eqref{def:vphi} above. Suppose $\st{(c_{00}, \norm{}_{\ell^1}, \norm{}_i)}{i < \omega}$ is a family of structures in $\mcal{C}$ such that
 	\begin{linenomath*} 
 	\[\norm{}_1 \leq \norm{}_2 \leq \cdots \leq \norm{}_n \leq \cdots\]
 	\end{linenomath*} 
 	and the $\vphi$-type $t = \lim_{i \to \mcal{U}}\mr{ltp}_{\vphi, \norm{}_i}$ is realized by $(c_{00}, \norm{}_{\ell^1}, \norm{}_*)$ in $\mr{Str}(L)$, then $\st{\mr{ltp}_{\vphi, \norm{}_i}}{i < \omega}$ uniquely determines $\norm{}_*$ over $\mcal{C}$. In particular, the Tsirelson norm is uniquely determined by its $\vphi$-type over $\mcal{C}$.
 \end{prop}

 The following result is a corollary in \cite{Casazza} and is purely analytical. The core analytical result is stated identically in \cite{CDI}.

 \begin{prop}\label{analysisCI}
 	Let $\norm{}_i$ be the $i$-th iterate in the construction of the Tsirelson norm. Then the following hold:
 	\begin{enumerate}[label=(\roman*)]
 		\item $\sup_{\norm{x}_{\ell^1}=1}\frac{\norm{x}_i}{\norm{x}_j} \leq 1$ for $i < j$.

 		\item $\sup_{\norm{x}_{\ell^1}=1}\frac{\norm{x}_i}{\norm{x}_j} \geq j$ for $i > j$.
 	\end{enumerate}
 	Thus, $\sup_{i<j}\vphi(\norm{}_i, \norm{}_j) \neq \inf_{j<i}\vphi(\norm{}_i, \norm{}_j)$.
 \end{prop}

Proposition \ref{analysisCI} and Theorem \ref{maindlc} yield:

 \begin{thm}
 	Let $L$ be a language of pairs of structures, and $\mcal{C}$ the class of structures $(c_{00}, \norm{}_{\ell^1}, \norm{})$ such that the norm completion of $(c_{00}, \norm{})$ is a Banach space including $\ell^p$ or $c_0$, and let $\norm{}_T$ be the Tsirelson norm. Let $\vphi$ be the formula as in \eqref{def:vphi} above. If the space $S^l_\vphi(\mcal{C})$ of $\vphi$-types over $\mcal{C}$ is \groth and countably tight, then $\norm{}_T$ is uniquely determined by its $\vphi$-type over $\mcal{C}$ and that $\vphi$-type is not explicitly definable over $\mcal{C}$.
 \end{thm}

 \begin{rmk*}
 	Notice that here we are dealing with languages for pairs of structures in which only first-order variables are considered. In such a context, definitions like Definition \ref{def:uniquelydetermined} are sound. An interesting logic with countably many formulas is second-order logic. Although it seems as if the topology is confined to subspaces of $2^\omega$, the model theory is not very well developed, even though Chang and Keisler \cite{Chang1973} have the development of the model theory of second-order logic as one of the problems in their classic book. However, see \cite{Hyttinen2012}. It is unclear to the authors whether the results and even the definitions presented here apply to second-order logic, which has been called ``set theory in sheep's clothing'' \cite{Quine1986}. In particular, it would be interesting to know about the definability of Tsirelson space in second-order arithmetic, which is often conflated with Analysis.
 \end{rmk*}

 We now prove a positive definability result which generalizes the one that appears in \cite{Casazza}. The proof is a variation of the original, in which their use of the Stone-Weierstrass Theorem is replaced by working with the double (ultra)limit condition. This is important because the Stone-Weierstrass Theorem is not nicely extendable beyond compact spaces.

 \begin{defi}
 	A structure $(c_{00}, \norm{}_{\ell^1}, \norm{}, e_0, e_1, \ldots)$ where $\norm{}$ is an arbitrary norm and $\seq{e}{n}{n<\omega}$ is the standard vector basis of $c_{00}$ is called a \emph{structure based on $c_{00}$}.
 \end{defi}

  \begin{thm}
 	Let $L$ be a language for pairs of structures, $\vphi$ the formula defined by \eqref{def:vphi} above, and $C$ a subclass of the class of structures based on $c_{00}$ such that every closed subspace of a space in $\clsr{C}$ includes a copy of $c_0$ or $\ell^p$. Assume that the space of $\vphi$-types over $C$ is \groth and countably tight. If the $\vphi$-type of $\mf{M}$ is explicitly definable from $C$, then $\mf{M}$ includes a copy of $c_0$ or $\ell^p$.
 \end{thm}

 \begin{proof}
 	Suppose that the global $\vphi$-type $\tau$ of $\mf{M}$ is \textit{approximated} by $\seq{\mf{M}}{i}{i < \omega} \subseteq C$ and $\mcal{U} \in \beta\omega$, that is
 	\begin{linenomath*} 
 	\[\tau(\lim_{j \to \mcal{V}}\mr{rtp}_{j, \mf{N_j}}) = \lim_{i \to \mcal{U}}\lim_{j \to \mcal{V}}\vphi(\mf{M}_i, \mf{N}_j).\]
 	\end{linenomath*} 
 	We show that there is an $\mf{N} \in \clsr{C}$ such that $\vphi(\mf{M}, \mf{N}) > 0$. Otherwise, $\vphi(\mf{M}, \mf{N}) = 0$ for each $\mf{N} \in \clsr{C}$. Since $\tau$ is explicitly definable, we exchange the order of the ultralimits by taking $\tau(\lim_{j \to \mcal{V}}\mr{rtp}_{j, \mf{N_j}})=\lim_{j \to \mcal{V}}\tau(\mr{rtp}_{j, \mf{N_j}})$ and then obtain
 	\begin{linenomath*} 
 	\[\lim_{i \to \mcal{U}}\lim_{j \to \mcal{V}}\vphi(\mf{M}_i, \mf{N}_j) = \lim_{j \to \mcal{V}}\lim_{i \to \mcal{U}}\vphi(\mf{M}_i, \mf{N}_j) = \lim_{j \to \mcal{V}}\vphi(\mf{M}, \mf{N}_j) = 0.\]
 	\end{linenomath*} 
 	But then consider $\vphi(\mf{M}_i, \mf{M}_j)$ and, by the symmetry of $D(\mf{M}, \mf{N})$, obtain $\vphi(\mf{M}, \mf{M}) = 0$, a contradiction. It is not hard to see that $\vphi(\mf{M}, \mf{N}) > 0$ implies that $D(\mf{M}, \mf{N}) < \infty$, which implies that $e^{\mf{M}}_n \mapsto e^{\mf{N}}_n$ determines an isomorphism (it is a linear bijection between the bases). This concludes the proof since $\mf{N}$ includes a copy of $c_0$ or $\ell^p$.
 \end{proof}
 
 \begin{rmk*}
 Notice that regardless of the results proved above, a space can have a copy of $c_0$ or $\ell^p$ while failing to be explicitly definable. 
 \end{rmk*}

 Let $L$ be a language for metric structures. The \emph{$\mcal{L}_{\omega_1, \omega}(L)$-formulas} are defined recursively as follows:
 \begin{enumerate}[label=(\alph*)]
 	\item All first-order $L$-formulas are $\mcal{L}_{\omega_1, \omega}(L)$-formulas.

 	\item If $\vphi_1, \ldots, \vphi_n$ are $\mcal{L}_{\omega_1, \omega}(L)$-formulas, and $g:[0, 1]^n \to [0, 1]$ is continuous, then $g(\vphi_1, \ldots, \vphi_n)$ is an $\mcal{L}_{\omega_1, \omega}(L)$-formula.

 	\item If $\st{\vphi_n}{n < \omega}$ is a family of $\mcal{L}_{\omega_1, \omega}(L)$-formulas, then $\inf_n\vphi_n$ and $\sup_n\vphi_n$ are $\mcal{L}_{\omega_1, \omega}(L)$-formulas. These can also be denoted by $\Meet_n\vphi_n$ and $\lJoin_n\vphi_n$ respectively.

 	\item If $\vphi$ is an $\mcal{L}_{\omega_1, \omega}(L)$-formula and $x$ is a variable then $\inf_x\vphi$ and $\sup_x\vphi$ are $\mcal{L}_{\omega_1, \omega}(L)$-formulas.
 \end{enumerate}

 An interesting feature of continuous $\mcal{L}_{\omega_1, \omega}$, noted in \cite{Eagle2015}, is that negation ($\neg$) becomes available in the classical sense: if $L$ is a language and $\vphi$ is an $\mcal{L}_{\omega_1, \omega}(L)$-formula, then we can define $\psi(x) = \lJoin_n\{\vphi(x) + \frac{1}{n}, 1\}$. Then $\mf{M} \ent \psi(a)$ if and only if there is an $n < \omega$ such that $\mf{M} \ent \max\{\vphi(a) + \frac{1}{n}, 1\}$ ; that is, $\max\{\vphi(a) + \frac{1}{n}, 1\} = 1$, which is the same as $\vphi(a) \leq 1 - \frac{1}{n}$, i.e.~$\mf{M} \not\ent \vphi(a)$. Thus, $\mf{M} \ent \psi(x)$ if and only if $\mf{M} \not\ent\vphi(x)$, and so $\psi$ corresponds to $\neg\vphi$.

 It is sometimes useful to restrict one's attention to (countable) fragments of $\mcal{L}_{\omega_1, \omega}(L)$, which are easier to work with than the full logic.

 Let $L$ be a language. A \emph{fragment} $\mcal{F}$ of $\mcal{L}_{\omega_1, \omega}(L)$ is a set of $\mcal{L}_{\omega_1, \omega}(L)$-formulas satisfying:
 \begin{enumerate}[label=(\alph*)]
 	\item Every first-order formula in in $\mcal{F}$.
 	\item $\mcal{F}$ is closed under finitary conjunctions and disjunctions.
 	\item $\mcal{F}$ is closed under $\inf_x$ and $\sup_x$.
 	\item $\mcal{F}$ is closed under subformulas.
 	\item \label{rm:doesnotwork}$\mcal{F}$ is closed under substituting terms for free variables.
 \end{enumerate}

 Notice that every subset of $\mcal{L}_{\omega_1, \omega}(L)$ generates a fragment, and that every finite set of formulas generates a countable fragment. There are two arguments to support the idea of working with countable fragments of $\mcal{L}_{\omega_1, \omega}(L)$. First, notice that a given proof involves only finitely many formulas which then generate a countable fragment of $\mcal{L}_{\omega_1, \omega}(L)$. Notice that even if we adjoined to our proof system some sort of $\omega$-rule, we would not have more than countably many formulas involved in a given proof. These formulas would still generate a countable fragment. Second (as noted independently by C.~Eagle in a personal communication with the authors), if an object is definable in $\mcal{L}_{\omega_1, \omega}(L)$, then it is definable in a countable fragment by the same argument. This remark does not hold for $\mcal{L}_{\omega_1, \omega_1}$ because of \ref{rm:doesnotwork} above.

 As the work of Casazza and Iovino deals with continuous logics which are finitary in nature, it was natural to ask whether their result on the undefinability of Tsirelson's space could be proved for continuous $\mcal{L}_{\omega_1, \omega}$, which is arguably a more natural language from the point of view of Banach space theorists.

 In discrete model theory, countable fragments of $\mcal{L}_{\omega_1, \omega}$ have been studied previously; for instance, as we mentioned previously, M.~Morley \cite{Morley1974} showed that the space of types of a countable fragment of $\mcal{L}_{\omega_1, \omega}$ is Polish.

 \begin{rmk*}
 	In the continuous case, we note that the space of types of a countable fragment $\mcal{F}$ of continuous $\mcal{L}_{\omega_1, \omega}(L)$ can be seen as a subspace of $[0, 1]^{\mcal{F}}$ which is metrizable and second countable, and so it is separable and first countable. Hence it is \groth and countably tight and thus our results apply.
 \end{rmk*}

 \section{The NIP and the Bourgain-Fremlin-Talagrand Dichotomy}
 A series of papers by P.~Simon \cite{Simon2015}, K.~Khanaki \cite{Khanaki2014,Khanaki2015}  and Khanaki and A.~Pillay \cite{Khanaki2017} has explored connections between the \emph{Non Independence Property} (NIP) and the famous dichotomy of \cite{Bourgain1978}:

 \begin{thm}
 	Let $X$ be a Polish space, and let $\seq{f}{n}{n < \omega}$ be a sequence in $C_p(X)$ with an accumulation point $f \in \reals^X$. Then one of the following holds:
 	\begin{enumerate}[label=(\roman*)]
 		\item $\seq{f}{n}{n < \omega}$ has a subsequence converging to $f$. 
 		\item The closure $\clsr{\seq{f}{n}{n < \omega}}$ in $\reals^X$ includes a copy of $\beta\omega$.
 	\end{enumerate}
 \end{thm}

 The informal conclusion is that such a sequence is either ``good'' (has a convergent subsequence) or is very complicated (its closure has cardinality $2^{2^{\aleph_0}}$---the number of ultrafilters on $\omega$). Although this dichotomy is widely attributed to \cite{Bourgain1978}, given that their work contains all the ideas for it, this more modern statement and a proof of it can be found in \cite{TodorcevicTopics}. The model-theoretic translation of this result is presented in detail in \cite{Simon2014}.

 At the suggestion of the editor, we have started looking at applications of $C_p$-theory to NIP. As opposed to the double limit conditions this paper deals with, the dichotomy above is not so easily generalized from the topological point of view. However, Haim Horowitz and the first author did realize that the dichotomy works for analytic spaces, i.e.~continuous images of the Baire space $\omega^\omega$ and, more generally, that if the dichotomy holds for a space $X$, then it holds for all continuous images of $X$. The proof goes as follows: suppose $f:X\to Y$ is continuous and that the dichotomy holds for $X$. Since the dichotomy is really a property of $\mathbb{R}^X$ and by Lemma \ref{lem:dual}, $\mathbb{R}^Y$ is homeomorphic to a subspace of $\mathbb{R}^X$ through the dual map $\Phi_f$, which also preserves continuity, then $Y$ has the dichotomy. 

 Todorcevic and Horowitz \cite{Horowitz} have shown that there is no hope for it to hold for \lind $p$-spaces as it fails for separable metrizable spaces: any maximal almost disjoint (MAD) family is a counterexample.

\begin{defi}
	An \emph{almost disjoint} family of subsets of $\omega$ is one in which any two members have finite intersection. The family is \emph{MAD} if it is maximal, i.e.~there is no bigger such family.
\end{defi}

It is not hard to show there is such a MAD family, since by Zorn's Lemma, any almost disjoint family can be extended to a maximal one. Moreover, if $V=L$, a co-analytic MAD family can be found \cite{Horowitz} to serve as a counterexample. Topologists like to consider almost disjoint families of rationals obtained by considering sequences of rationals converging to different irrationals.

\begin{thm}[\cite{Horowitz}]
	Let $A$ be a MAD family, considered as a subspace of the Cantor set. Then the BFT dichotomy fails for $A$. 
\end{thm}

\begin{proof}[Sketch of Proof]
	Consider the following family in $C_p(\mcal{A})$: for each $n < \omega$, let $f_n: \mcal{A} \to 2$ be defined by $f_n(x) = 1$ if and only if $n \in x$. Then the zero function $\mathbf{0}$ is in the closure $\clsr{\seq{f}{n}{n<\omega}}$. By MADness, there is no converging subsequence. One can also show that the closure $\clsr{\seq{f}{n}{n<\omega}}$ consists precisely of the $f_n$'s, the function $\mathbf{0}$, and the characteristic functions of each member of $\mcal{A}$, i.e.~the functions defined by $f_y(x) = 1$ if and only if $x = y$ for each $y \in \mcal{A}$.
\end{proof}

\emph{Angelic spaces} were introduced by Fremlin and appeared in Pryce \cite{Pryce1971}. They are defined to be spaces $X$ such that for every subset $A$ of $X$, if $A$ is countably compact in $X$, then $A$ is relatively compact and Fr\'echet-Urysohn. Then we can restate Proposition \ref{prop:gspace} as: a space is a hereditary $g$-space if and only if it is angelic. In their seminal paper, Bourgain, Fremlin, and Talagrand proved that if $X$ is Polish, then the first Baire class $B_1(X)$ (the class of pointwise limits of continuous real-valued functions on $X$) is angelic with the subspace topology inherited from the product topology on $\reals^X$. In his lecture at the Moscow conference celebrating the 80th birthday of Prof.~A.~V.~\arhan, the second author suggested that topologists investigate $B_1(X)$ with the same vigour with which they have studied $C_p(X)$. This will likely provide new insights into NIP. We plan to do this ourselves.

Bourgain, Fremlin, and Talagrand prove their results for $K$-analytic spaces, which, as mentioned before, are continuous images of the perfect pre-images of separable completely metrizable spaces. Thus they are the subclass of \lind $\Sigma$-spaces obtained by starting with separable completely metrizable spaces rather than arbitrary separable metrizable ones. Indeed, MAD families cannot be completely metrizable, or even analytic. This is a classical result. There is a ``BFT-flavoured'' proof in \cite{Horowitz}. Thus a topological approach to NIP should start with $K$-analytic spaces. It is an interesting question whether one can go beyond that class. The second author is engaged in a project to define \emph{$K$-projective} sets in analogy to the projective sets in descriptive set theory \cite{TallGrot}, but it is not clear whether this is possible. If it is, does this enable NIP results to be extended to a larger class of theories?

Connections between NIP and compact spaces are explored in \cite{Simon2014}. There is a brief survey of references for related topology in \cite{Hamel}. There is much more to be investigated, but that is beyond the scope of this paper.

\section{Appendix: Proof that all \lind $\Sigma$-spaces are \groth}
First we need the \textit{\arhan-Pytkeev Theorem} (\cite{ArhangelskiiFunction}) and some definitions.

\begin{thm}\label{thm:tcpsup}
	$t(C_p(X)) = \sup\{L(X^{n+1}) : n < \omega\}$.
\end{thm}

\begin{proof}
	Suppose first that $L(X^n) \leq \kappa$ for all $n < \omega$. Let $A \subseteq C_p(X)$ and take $f \in \clsr{A} \subseteq C_p(X)$. Now, for each $n < \omega$ and each $\xi = (x_0, \ldots, x_n) \in X^{n+1}$, there is some $g_\xi \in A$ such that
	\begin{linenomath*}
	\[(\forall i \leq n) (\abs{g_\xi(x_i) - f(x_i)} < \frac{1}{n+1}).\]
	\end{linenomath*}
	Since both $g_\xi$ and $f$ are continuous, for each $i \leq n$ there is a neighbourhood $O(x_i)$ of $x_i$ such that $(\forall y \in O(x_i))(\abs{g_\xi(y) - f(y)} < \frac{1}{n+1})$.

	Now define the neighbourhood $U_\xi = O(x_0) \times \cdots \times O(x_n)$ and let $\mcal{U}_{n+1} = \st{U_\xi}{\xi \in X^{n+1}}$. Clearly, $\xi \in U_\xi$ and so $\mcal{U}_{n+1}$ covers $X^{n+1}$. Since $L(X^{n+1}) \leq \kappa$, we have some $\mcal{U}^*_{n+1} \subseteq \mcal{U}_{n+1}$, a cover of $X^{n+1}$, such that $\crd{\mcal{U}^*_{n+1}} \leq \kappa$, and we can define $B_{n+1} = \st{g_\xi}{U_\xi \in \mcal{U}^*_{n+1}}$, for each $n < \omega$. Now let $B = \Union_{n < \omega}B_{n+1}$. Clearly, $\crd{B} \leq \kappa$ and $B \subseteq A$. We now claim that $f \in \clsr{B}$. Indeed, let $y_0, \ldots, y_n \in X$ and $\epsi > 0$, and consider the basic open neighbourhood $W(f;y_0, \ldots, y_n;\varepsilon)$; we see that it has non-empty intersection with $B$. Without loss of generality, suppose $\frac{1}{n+1} < \epsi$. Since $\mcal{U}^*_{n+1}$ is a cover of $X^{n+1}$, let $\xi \in X^{n+1}$ be such that $(y_0, \ldots, y_n) \in U_\xi$. Then $(\forall i \leq n)\abs{g_\xi(y_i) - f(y_i)} < \frac{1}{n+1} < \epsi$.

	The converse is more difficult to prove. We do not actually need it for our proof that \lind $\Sigma$ implies Grothendieck. Nonetheless, for the sake of completeness and exposure to $C_p$-theory, here it is. Suppose that $t(C_p(X)) \leq \kappa$ and fix $n < \omega$. Let $\mcal{V}$ be an open cover of $X^{n+1}$. Let $\mcal{E}$ be the family of all finite \emph{$\mcal{V}$-small} families, where a family $\mcal{W}$ of open subsets of $X$ is \emph{$\mcal{V}$-small} if for any $W_0, \ldots, W_n \in \mcal{W}$, there is some $V \in \mcal{V}$ such that $W_0\times \cdots \times W_n \subseteq V$. For each $\mcal{W} \in \mcal{E}$, define
	\begin{linenomath*}  
	\[A_\mcal{W} = \left\{f \in C_p(X) : f\left(X\setminus \Union \mcal{W}\right) = \{0\}\right\}\]
	\end{linenomath*} 
	 and let $A = \Union_{\mcal{W} \in \mcal{E}}A_\mcal{W}$. We claim that $\clsr{A} = C_p(X)$.

	 Indeed, let $f \in C_p(X)$ and let $K$ be a finite subset of $X$. Take a finite family $\mcal{Q}_K$ of open sets in $X$ with the following property: for any $(y_0, \ldots, y_n) \in K^{n+1}$, there are $Q_0, \ldots, Q_n \in \mcal{Q}_K$ such that $y_i \in Q_i$ for each $i \leq n$ and $Q_0 \times \cdots \times Q_n \subseteq G$ for some $G \in \mcal{V}$. Clearly, $K \subseteq \Union \mcal{Q}_K$. For each $x \in K$, define $W_x = \Insect\{V \in \mcal{Q}_K : x\in V\}$. Now, let $\mcal{W}_K = \st{W_x}{x \in K}$. Again, $K \subseteq \Union \mcal{W}_K$. We now see that $\mcal{W}_K$ is $\mcal{V}$-small: consider a set of the form $W_{x_0} \times \cdots \times W_{x_n}$. By construction, there are $Q_0, \ldots, Q_n \in \mcal{Q}_K$ such that $(\forall i \leq n)(x_i \in Q_i)$ and $Q_0 \times \cdots \times Q_n \subseteq G$ for some $G \in \mcal{V}$. Since $W_{x_i} \subseteq Q_i$ for each $i \leq n$, we have $W_{x_0}\times \cdots \times W_{x_n} \subseteq G$, proving the claim.

	 Now, by complete regularity, take a function $g \in C_p(X)$ such that $\rstrct{f}{K} = \rstrct{g}{K}$ and $g\left(X\setminus\Union\mcal{W}\right) = \{0\}$. The point is that $K$ is finite and $K \subseteq \Union \mcal{W}_K$. So we can certainly get a continuous real-valued function $h$ which is $1$ on $K$ and $0$ on $X\setminus \Union\mcal{W}$. Then we can add some continuous real-valued functions to $h$ to get the desired values on $K$, since a finite sum of continuous functions is continuous. Clearly, $g \in A_{\mcal{W}_K} \subseteq A$ and $g$ lies in all basic neighbourhoods of $f$ based on $K$ (i.e. basic neighbourhoods of the form $W(f;x_0,...,x_n,\epsi)$ where $x_0,...,x_n\in K$ and $\epsi>0$). Thus, $f \in \clsr{A}$ and $\clsr{A} = C_p(X)$.

	 If we take $\mathbf{1}$ to be the function identically $1$ we have $\mathbf{1} \in \clsr{A}$. Since $t(C_p(X)) \leq \kappa$, there is a $B \subseteq A$ such that $\mathbf{1} \in \clsr{B}$ and $\crd{B} \leq \kappa$. Then, there is a subfamily $\mcal{E}_0 \subseteq \mcal{E}$ for which $B \subseteq \Union_{\mcal{W} \in \mcal{E}_0}A_\mcal{W}$ and $\crd{\mcal{E}_0} \leq \kappa$. Given any $\mcal{W} \in \mcal{E}_0$, we introduce the following notation: for each $V = (Q_0, \ldots, Q_n) \in \mcal{W}^{n+1}$, we fix $G_V \in \mcal{V}$ such that $V \subseteq G_V$ and then define $\mcal{V}_\mcal{W} = \st{G_V}{V \in \mcal{W}^{n+1}}$. Since the family $\mcal{V}_\mcal{W}$ is finite and $\crd{B}\leq \kappa$, then $\mcal{V}^* = \Union_{\mcal{W}\in\mcal{E}_0}\mcal{V}_\mcal{W}$  always has cardinality $\leq \kappa$.

	 It remains to show that $\mcal{V}^*$ covers $X^{n+1}$. Take $(x_0, \ldots, x_n) \in X^{n+1}$ and let $U = \st{f_p(X)}{(\forall i \leq n)(f(x_i)>0)}$. Notice that $U$ is open in $C_p(X)$ and that $\mathbf{1} \in U$. Since $\mathbf{1} \in \clsr{B}$ and $B \subseteq \Union_{\mcal{W} \in \mcal{E}_0}A_\mcal{W}$, there is a $\mcal{W}_0 \in \mcal{E}_0$ such that $U \insect A_{\mcal{W}_0} \neq \emptyset$. Now take $g \in U \insect A_{\mcal{W}_0}$; then $g(x_i) > 0$ for each $i \leq n$ and $g(x) = 0$ for $x \in X\setminus \Union \mcal{W}_0$. Thus, $x_i \in \Union \mcal{W}_0$ for each $i \leq n$. Take $Q_i \in \mcal{W}_0$ such that $x_i \in Q_i$, $i \leq n$. Then $(x_0, \ldots, x_n) \in Q = Q_0 \times \cdots \times Q_n \subseteq G_V \in \mcal{V}_{\mcal{W}_0} \subseteq \mcal{V}^*$.
\end{proof}

Our plan is to show \lind $p$-spaces are weakly \groth and then use Proposition \ref{prop:gspace} to show they and (hence) \lind $\Sigma$-spaces are \groth. First, we introduce some important definitions.

\begin{defi}
	A topological space is a \emph{$k$-space} if subspaces are closed if and only if their intersection with each compact subspace is closed.
\end{defi}

We will show

\begin{prop}[\cite{Arhangelskii1997a}]\label{prop:kspacewgroth}
	Every $k$-space is weakly \groth.
\end{prop}

\begin{thm}\label{lindpsk}
	Every \lind $p$-space is a $k$-space.
\end{thm}

Actually, the ``\lind'' is superfluous, but it is all we need and its inclusion makes the proof much easier. Metrizable spaces are easily seen to be $k$, since if a subspace $W$ is not closed, there is a sequence converging from $S$ to a point outside. That sequence together with its limit are a compact set with non-closed intersection with $S$. So all we have then to do is show:

\begin{lem}
	A perfect pre-image of a $k$-space is a $k$-space.
\end{lem}

This is Theorem 3.7.25 in \cite{Engelking1989}. We shall sketch his proof, but it's not $C_p$-theory, so the reader is welcome to skip it.

\begin{proof}
Given an arbitrary Hausdorff space $X$, one can put a finer topology on $X$ by declaring closed all subsets of $X$ that have closed intersections with each compact subspace of $X$. This new space is called $kX$ and can be shown to be a $k$-space. Every continuous map $f$ from $X$ to a Hausdorff space $Y$ induces a continuous map $kf$ from $kX$ to $kY$, by assigning to $x \in kX$ the image $f(x) \in kY$. It is not hard to show that $kf$ is perfect since $f$ is. Define a function $k_X: kX \to X$ by $k_X(x) = x$. Then $kf = f\circ k_X$ and it follows that $k_X$ is perfect. But $k_X$ is one-to-one, so it's a homeomorphism, and so $X$ is a $k$-space. 
\end{proof}

We shall also need:

\begin{lem}[{\cite[3.7.25]{Engelking1989}}]
	A mapping from a $k$ space $X$ to a space $Z$ is continuous if and only if for any compact $K$ included in $X$, the restriction $f|Z$ is continuous.
\end{lem}

\begin{proof}
	Take any closed $B \subseteq Y$. For any compact $K \subseteq X$, we have $f^{-1}(B) \insect K = (f|K)^{-1}(B)$, which shows $f^{-1}(B)$ is closed, since $X$ is a $k$-space. Thus $f$ is continuous.
\end{proof}

The following result is well-known:

\begin{prop}\label{contcount}
The continuous image of a countably compact space is countably compact.
\end{prop}

\begin{proof}
Let $X$ be countably compact and $f:X\to Y$ a continuous function. Suppose $B$ is an infinite subset of $f(X)$, and let $X'$ be obtained by picking one point from each $f^{-1}(\{b\})$, for each $b \in B$. Then $X'$ is infinite and so has a limit point $x \in X$. Then $f(x)$ is a limit point of $B$, since if $U$ is an open set about $f(x)$, the pre-image $f^{-1}(U)$ is an open set about $x$ and so contains some $a \in X'$. Then $U$ contains $f(a)$, which is in $B$.
\end{proof}

\begin{proof}[Proof of Proposition \ref{prop:kspacewgroth}]
	Let $A$ be countably compact in $C_p(X)$, with $X$ a $k$-space. Since $C_p(X)$ is canonically embedded in $\reals^X$, $A$ is also a subspace of $\reals^X$. By continuity, by Proposition \ref{contcount} the images of $A$ under the natural projections of $\reals^X$ onto $\reals$ are also countably compact in $\reals$. Since $\reals$ is metrizable, it is normal so the images of $A$ have countably compact closures, but then by metrizability they have compact closures. Since $A$ is included in the topological product of these projections, the closure of $A$ in $\reals^X$ is a compact subspace $F$ of $\reals^X$. It remains to show that $F \subseteq C_p(X)$, i.e.~ each member of $f$ is continuous. 

	The restriction map $r$ mapping $\reals^X$ into $\reals^Y$ sends $A$ to a dense subset of $r(F)$, since $r$ is continuous and $A$ is dense in $F$. Since $r$ is continuous, $r(A)$ is countably compact in $C_p(Y)$ by Proposition \ref{contcount}. As a compact space, $Y$ is \groth and so the closure of $r(A)$ in $C_p(Y)$ is compact. But then that closure must equal $r(F)$, since, by continuity of $r$, $r(F)$ is the closure of $r(A)$ in $\reals^X$, so is the smallest closed set including $r(A)$ there. Thus, for each $f \in F$ and each compact $Y \subseteq X$, there is a $g$ in the closure of $A$ such that the restriction of $f$ to $Y$ coincides with the restriction of $g$ to $Y$. Since $X$ is a $k$-space, it follows that $f$ is continuous.
\end{proof}

\begin{defi}
Let $X$ be a topological space.
\begin{enumerate}[label=(\alph*)]
    \item The \textit{density} $d(X)$ of $X$ is the smallest cardinality of a dense subset of $X$.
    \item A \textit{network} of $X$ is a family $\mathcal{A}$ of subsets of $X$ such that given any $x\in X$ and open set $U$ containing $x$, there is an $A\in\mathcal{A}$ satisfying $x\in A\subseteq U$. The \textit{network weight} $nw(X)$ of $X$ is the smallest cardinality of a network of $X$.
    \item The \textit{i-weight} $iw(X)$ of $X$ is the minimal weight of all the spaces onto which $X$ can be continuously mapped. 
\end{enumerate}
\end{defi}

We see that if $X$ is compact, then $nw(X)=w(X)$ (see \cite{ArhangelskiiFunction}): given any Hausdorff space $X$, we can construct a Hausdorff space $Y$ such that $w(Y)\leq nw(X)$ and $Y$ is a continuous one-to-one image of $X$. To see this, let $\mathcal{A}$ be a network of $X$ of size $nw(X)$. Consider all the pairs of elements $A_1,A_2\in\mathcal{A}$ for which there are disjoint open subsets $U_1$ and $U_2$ such that $A_1\subseteq U_1$ and $A_2\subseteq U_2$. It suffices to define the topological space $Y$ as the one with the same underlying set as $X$ and with the topology generated by the all the open sets of $X$ obtained as above. Thus the identity map is as desired. That $nw(X)=w(X)$ for every compact space $X$ follows from the preceding discussion and the well-known fact that a continuous one-to-one map between a compact space and a Hausdorff space is a homeomorphism.  

Some other relatively easy facts regarding these cardinals are the following: $iw(X)\leq nw(X)$, $d(X)\leq w(X)$, $nw(X)=nw(C_p(X)$ and $iw(X)=d(C_p(X))$ for every $X$. See \cite{ArhangelskiiFunction}, p.~26, for a detailed exposition. Now we proceed with two definitions that arise naturally when studying topological function spaces:

\begin{defi}
Let $X$ be a topological space.
\begin{enumerate}[label=(\alph*)]
    \item $X$ is \textit{$\kappa$-monolithic} if $nw(\overline{A})\leq \kappa$ whenever $\abs{A}\leq \kappa$. We say that $X$ is \textit{monolithic} if it is $\kappa$-monolithic for every $\kappa$.
    \item $X$ is \textit{$\kappa$-stable} if for every continuous image $Y$ of $X$ we have: $iw(Y)\leq \kappa$ if and only if $nw(Y)\leq \kappa$. We say that $X$ is \textit{stable} if it is $\kappa$-stable for every $\kappa$.
    
\end{enumerate}
\end{defi}
Notice that being monolithic is inherited by all subspaces. Examples of monolithic spaces include all metric spaces and spaces with a countable network, whereas examples of stable spaces include all compact spaces and \lind $\Sigma$-spaces (as we shall prove below). The following theorem shows that stability and monolithicity are intertwined notions for $X$  and $C_p(X)$. 

\begin{thm}[\arhan]\label{monosta}
$X$ is $\kappa$-stable if and only if $C_p(X)$ is $\kappa$-monolithic. 
\end{thm}
\begin{proof}
Assume $C_p(X)$ is $\kappa$-monolithic and let $f:X\to Y$ be a continuous surjection. Since always $iw(Y)\leq nw(Y)$, we also assume that $iw(Y)\leq \kappa$. By Lemma \ref{lem:rstrct}, $C_p(Y)$ is homeomorphic to a subspace of $C_p(X)$ and so it's $\kappa$-monolithic. In particular, by looking at a dense subset of $C_p(Y)$ and using the definition of $\kappa$-monolithicity, $nw(C_p(Y))\leq d(C_p(Y))$. Using that $iw(Y)=d(C_p(Y))$ and $nw(Y)=nw(C_p(Y)$, we obtain $nw(Y) \leq iw(Y)$ as desired. Conversely, assume that $X$ is $\kappa$-stable and let $M\subseteq C_p(X)$ be such that $\abs{M}<\kappa$. Define $f:X\to \reals^M$ as the diagonal product of $M$, i.e. $f(x)=\langle g(x) : g\in M\rangle$, and let $Y=f(X)$. By our choice, $w(Y) \leq \abs{M}\leq \kappa$ as in the proof of Theorem \ref{thm:crds}.

The rest of the proof is a bit dense. Although the theorem is important, the reader is unlikely to use the \emph{real quotient topology}, so it is not unreasonable to take the rest on faith. But we shall proceed.

Let $Y^{*}$ be the topological space with the same underlying set as $Y$ but endowed with the \textit{real quotient topology corresponding to $f$}, i.e. the topology on $Y^*$ is the strongest completely regular topology for which $f$ is continuous (see \cite{ArhangelskiiFunction}, p.~14). Then the identity $i:Y^{*}\to Y$ is a continuous bijection. By the definition of the i-weight, $iw(Y^{*})\leq w(Y) \leq \kappa$. 

Conversely, since $f$ defines a continuous surjection from $X$ onto $Y^{*}$, i.e. the real quotient map $f^{*}=i^{-1}\circ f$, and $X$ is $\kappa$-stable, $nw(Y^{*})\leq \kappa$ and $nw(C_p(Y^{*}))=nw(Y^{*})\leq \kappa$. Also, the space  $C_p(Y^*)$ is homeomorphic to the closed subspace $F=\Phi_{f^*}[C_p(Y^*)]=\{g\circ f^* : g\in C_p(Y^*)\}$ of $C_p(X)$ since $f^*$ is a real quotient map (see Section 0.4 of \cite{ArhangelskiiFunction} for a detailed discussion of real quotient maps). Now let $g\in M$ and notice that $g=\pi_g\circ f=\pi_g \circ i \circ f^*$ where $\pi_g:\reals^M\to \reals$ is the projection $\pi_g(\langle h(x) : h\in M\rangle)=g(x)$. Also, $\pi_g \circ i:Y^*\to \reals$ is continuous, so $g\in F$ and thus $M\subseteq F$. Thus $\overline{M}\subseteq F$ since $F$ is closed, and we conclude by noticing that $nw(\overline{M})\leq nw(F) = nw(C_p(Y^*))\leq \kappa$. 
\end{proof}

Although we will not use this fact here, it is worth noting that the previous theorem is self-dual, i.e.: $X$ is $\kappa$-monolithic if and only if $C_p(X)$ is $\kappa$-stable. See \cite{ArhangelskiiFunction}, p.~78, for a proof. 

The following diagram-chasing lemma appears in \cite{ArhangelskiiFunction}, p.~80, and is used to prove that \lind $\Sigma$-spaces are stable. We omit the proof here.

\begin{lem}
Let $f: X\to Y$, $g: X\to Z$ and $h: Z\to T$ be continuous surjections. Also assume that $f$ is perfect and that $h$ is a bijection. Then $Z$ is a continuous image of a closed subspace of $Y \times T$. 
\end{lem}

\begin{thm}\label{lindssta}
\lind $\Sigma$-spaces are stable.
\end{thm}
\begin{proof}
Let $Z$ be a \lind $\Sigma$-space. Since \lind $\Sigma$-spaces are continuous images of \lind $p$-spaces, any continuous image of $Z$ is also \lind $\Sigma$ so it suffices to see that $nw(Z)\leq iw(Z)$. \\

Suppose $iw(Z)=\kappa$ and let $h:Z\to T$ be a continuous bijection and suppose $w(T)\leq \kappa$. Again, by definition of \lind $\Sigma$-spaces, there is a second countable space $Y$, a space $X$, a perfect surjective map $f: X\to Y$ and a continuous surjection $g: X\to Z$. By the previous lemma, $Z$ is the continuous image of a closed subspace of $Y\times T$. Thus $nw(Z)\leq nw(Y\times T)\leq \kappa$ as desired. 
\end{proof}

\begin{prob}
    Is there a connection between $C_p$-theoretic stability (or monolithicity) and model-theoretic stability?
\end{prob}

\begin{lem}\label{ctmono}
A countably tight monolithic compact space is Fréchet-Urysohn. 
\end{lem}
\begin{proof}
Let $X$ be a countably tight monolithic compact space and suppose that $x\in\overline{A}$ for some $A\subseteq X$. Let $B\subseteq A$ be a countable subset such that $x\in \overline{B}$. Thus $w(\overline{B})=nw(\overline{B})=\aleph_0$ since $\overline{B}$ is compact and monolithic. Then $\overline{B}$ is second countable and so there is a sequence in $B$ converging to $x$. 
\end{proof}

We now have all we need to prove Theorem \ref{thm:lindSgroth} that says that all \lind $\Sigma$ spaces are Grothendieck:

\begin{proof}[Proof of Theorem \ref{thm:lindSgroth}]
 By Theorem \ref{lindpsk}, every \lind $p$-space is a $k$-space. Then by Proposition \ref{prop:kspacewgroth} and Theorem \ref{lindssta}, \lind $p$-spaces are weakly Grothendieck and stable. By the \arhan-Pytkeev Theorem, $C_p(X)$ for a \lind $p$-space $X$ is countably tight, since finite products of \lind $p$-spaces are \lind. Moreover, such $C_p(X)$ is monolithic by Theorem \ref{monosta}. We conclude that \lind $p$-spaces are Grothendieck by Proposition \ref{prop:gspace} and Lemma \ref{ctmono}. Finally, given that \lind $\Sigma$ spaces are continuous image of \lind $p$-spaces, we conclude that \lind $\Sigma$ spaces are Grothendieck by Theorem \ref{thm:contimage}.
\end{proof}

\begin{defi}
	A topological space $X$ is \emph{$\sigma$-compact} if it can be written as a countable union of compact subspaces. $X$ is \emph{k-separable} if it has a dense $\sigma$-compact subspace.
\end{defi}

The following theorem collects well-known (to $C_p$-theorists) results on conditions that imply Grothendieck:

\begin{thm}\label{whatgrot}
	A topological space $X$ is a \groth space if it satisfies any of the following:
	\begin{enumerate}[label=(\roman*)]
		\item $X$ has a dense countably compact subspace.
		\item $X$ is $k$-separable.
		\item \label{groth3}$X$ has a dense \lind $\Sigma$-space.
	\end{enumerate}
\end{thm}

\begin{proof}\mbox{}
\begin{enumerate}[label=(\roman*)]
	\item \groth's Theorem implies that countably compact spaces are weakly \groth. Pryce \cite{Pryce1971} proved that they are indeed \groth. The result follows from Theorem \ref{thm:densesubspace}.

	\item Follows from \ref{groth3} (see \cite{{Tkachuk2010}}) and Theorem \ref{thm:densesubspace}.

	\item This follows immediately from Theorems \ref{thm:densesubspace} and \ref{thm:lindSgroth}.
\end{enumerate}
\end{proof}

\begin{prob}
	Find a common generalization of these three conditions that implies \groth.
\end{prob}

\section{Acknowledgement}
We have greatly appreciated the comments of the anonymous referee, who has prodded us to make this paper more accessible to model theorists.

\newpage
\bibliographystyle{abbrv}
\bibliography{Beyond76.bib}

\end{document}